\documentclass{amsart}

\usepackage{amsmath}
\usepackage{amsfonts}
\usepackage{amssymb}
\usepackage{graphicx}
\usepackage{hyperref}
\usepackage{mathrsfs}
\usepackage{pb-diagram}
\usepackage[all]{xy}
\usepackage{color}
\usepackage{caption}
\usepackage{subcaption}

\newtheorem{theorem}{Theorem}[section]
\newtheorem{lemma}[theorem]{Lemma}

\newtheorem{prop}[theorem]{Proposition}
\newtheorem{defn}[theorem]{Definition}

\newtheorem{remark}[theorem]{Remark}

\numberwithin{equation}{section}

\newcommand{\Z}{\mathbb{Z}}
\newcommand{\R}{\mathbb{R}}
\newcommand{\C}{\mathbb{C}}

\newcommand{\bP}{\mathbb{P}}
\newcommand{\bT}{\mathbb{T}}

\newcommand{\cP}{\mathcal{P}}

\newcommand{\Conv}{\mathrm{Conv}}

\newcommand{\consti}{\mathbf{i}\,}

\newcommand{\bS}{\mathbb{S}}
\newcommand{\cA}{\mathcal{A}}

\begin{document}

\title{Affine elliptic surfaces with type-A singularities and orbi-conifolds}
\author[Lau]{Siu-Cheong Lau}
\address{Department of Mathematics and Statistics\\ Boston University}
\email{lau@math.bu.edu}

\begin{abstract}
Following the work of Castano-Bernard and Matessi on conifold transition in the Gross-Siebert program, we construct orbi-conifold transitions of the Schoen's Calabi-Yau threefold and their mirrors.  The construction glues together the local models for orbi-conifold transitions in the previous work with Kanazawa.
\end{abstract}

\maketitle

\section{Introduction}

Conifold transition is an important topic in the study of Calabi-Yau geometries and string theory.  By the discoveries of \cite{Reid,GH,Wang}, the moduli spaces of three-dimensional CY complete intersections in a product of projective spaces are connected by conifold transitions.  It was found by Friedman \cite{Friedman} and Smith-Thomas-Yau \cite{STY} that there exists non-trivial topological obstruction for conifold transition.  In the direction of mirror symmetry, HMS for the local resolved and deformed conifolds was proved by Chan-Pomerleano-Ueda \cite{CPU}.  The mirror of the Atiyah flop was constructed in a joint work of the auther with Fan-Hong-Yau \cite{FHLY} using a noncommutative mirror of the conifold.

More generally geometric transitions involve singularities of deeper levels than conifolds.   Singularities worse than conifolds may occur when 
several Lagrangian spheres vanish simultaneously.
For instance, generalized conifolds and orbifolded conifolds are two natural generalizations of the conifold.  In \cite{AKLM}, Aganagic-Karch-L\"ust-Miemiec deduced that these two classes of singularities are mirror to each other by gauge-theoretical methods.  In a joint work of the author with Kanazawa \cite{KL} mirror symmetry for these local singularities was realized using the SYZ program.  SYZ for general local Gorenstein singularities was derived in a previous work of the author \cite{L13} using the Lagrangian fibrations constructed by Gross \cite{Gross-eg} and Goldstein \cite{Goldstein}.  

A natural question is how to realize these geometric transitions and mirror symmetry in the compact setting.
For compact manifolds, Castano-Bernard and Matessi \cite{CM2} made a beautiful construction of conifold transitions and their mirrors using a symplectic version of the Gross-Siebert program that they developed earlier \cite{CM1}.  They studied the local affine geometries of the base of Lagrangian fibrations for local conifold transitions, and glued the local models of Lagrangian fibrations to a global geometry.

In this paper we construct compact generalized and orbifolded conifolds by extending the method of \cite{CM2}.  We construct the local affine geometries modeling the base of the SYZ fibrations on generalized and orbifolded singularities (and also their smoothings and resolutions).  Then we formulate global geometries that contain both of these singularities.  We call them to be orbi-conifolds.  

The discriminant locus of the global structure naturally contains orbifolded edges and orbifolded positive or negative vertices.  It gives an orbifold generalization of simple and positive tropical manifolds in the the Gross-Siebert program.

The Schoen's Calabi-Yau threefold \cite{Schoen} provides an excellent source of examples for orbi-conifolds.  The threefold is a resolution of the fiber product of two rational elliptic surfaces over the base $\bP^1$.  \cite{CM2} constructed compact conifolds which are degenerations of the Schoen's CY and their mirrors by using fan polytopes of toric blow-ups of $\bP^2$.   In this paper we treat all the reflexive polygons uniformly and construct compact orbi-conifolds and their mirrors.  The result is the following.

\begin{theorem}[Orbi-conifold degeneration of Schoen's CY] \label{thm:main}
Each pair of reflexive polygons $(P_1,P_2)$ (\emph{where $P_1$ and $P_2$ are not necessarily dual to each other}) corresponds to an orbifolded conifold degeneration $O^{(P_1,P_2)}$ of a Schoen's Calabi-Yau threefold, and also corresponds to a generalized conifold degeneration $G^{(P_1,P_2)}$ of a mirror Schoen's Calabi-Yau.  A resolution of $O^{(P_1,P_2)}$ is mirror to a smoothing of $G^{(\check{P}_1,\check{P}_2)}$ in the sense of Gross-Siebert, and vice versa.  ($\check{P}$ denotes the dual polygon of $P$.)
\end{theorem}

The connections between Calabi-Yau geometries and affine geometries play a key role.   Gross \cite{Gross-topo} found a beautiful topological realization of mirror symmetry using affine geometries with polyhedral decompositions.
   
Independently, Haase-Zharkov \cite{HZ1,HZ2,HZ3} found a brilliant construction of affine structures on spheres by using a pair of dual reflexive polytopes.  They are useful to study the geometries of Calabi-Yau complete intersections.

Based on these pioneering works, Gross-Siebert \cite{GS1,GS2,GS07} developed their celebrated program of toric degenerations and formulated an algebraic reconstruction of mirror pairs.  A symplectic version of the reconstruction was found by Castano-Bernard and Matessi \cite{CM1}.  The construction in this paper is a natural generalization of the work of \cite{CM2}.

To construct orbi-conifold degenerations of the Schoen's CY, first we study degenerations of rational elliptic surfaces via affine geometry.  

The connection between affine surfaces and symplectic four-folds is very well-understood and dates back to the early works of Symington \cite{Symington} and Leung-Symington \cite{LS}.  They defined the notion of an almost toric fourfold, which is a symplectic fourfold with a Lagrangian fibration under certain topological constraints.  They classified almost toric fourfolds by using the affine base of Lagrangian fibrations.

Rational elliptic surfaces form a subclass of almost toric fourfolds and hence are well understood by \cite{Symington,LS}.  Here we concern about singular surfaces with $A_n$-orbifold singularities, which will be important to the construction of orbi-conifolds.  We construct the following.

\begin{prop}[Rational elliptic surfaces with singularities] \label{thm:ell-intro}
Each reflexive polygon $P$ corresponds to two rational elliptic surfaces $S_P$ and $S_{P}'$ with type $A$ singularities, where the configurations of singularities depend on the integral properties of $P$.  $S_P - D_{S_P}$ and $S_{\check{P}}' - D_{S_{\check{P}}'}$ are mirror to each other in the sense that they are discrete Legendre dual to each other, where $D$ denotes an anti-canonical divisor.
\end{prop}

The topology of elliptic surfaces is closely related to the `12 Property' for a reflexive polygon $P$.  Namely the sum of affine lengths of edges and orders of vertices of $P$ must be $12$.  They correspond to the number of singular fibers (counted with multiplicities) of a rational elliptic surface. 

Indeed the `12 Property' holds for more general objects called legal loops \cite{PR}.  We find that they correspond to topological (which may not be Lagrangian) torus fibrations $M$ over $\bS^2$ which have been extensively studied by \cite{Matsumoto,Iwase}.  The multiple of $12$ corresponds to $-3\sigma/2  = - p_1 /2$ where $\sigma$ is the signature and $p_1$ is the first Pontryagin number of the real four-fold $M$.

The organization is as follows.  We review the SYZ construction for the related local singularities in Section \ref{sec:SYZ}.  Then we construct the local affine geometries modeling the singularities in Section \ref{sec:loc-aff}.  In Section \ref{sec:surf} we focus on the relation between rational elliptic surfaces and the `12' Property for polygons.  In Section \ref{sec:Schoen} we formulate the notion of orbi-conifold and construct orbi-conifold transitions of the Schoen's CY.

\subsection*{Acknowledgment}
The author is grateful to Ricardo Casta\~no-Bernard for introducing his beautiful joint work with Diego Matessi on Schoen's Calabi-Yau threefold.  He expresses his gratitude to Atsushi Kanazawa for useful discussions and comments.  He appreciates for the continuous encouragement of Naichung Conan Leung and Shing-Tung Yau.

\section{A quick review on the SYZ mirrors for local orbifolded conifolds} \label{sec:SYZ}

In \cite{KL}, we studied SYZ mirror symmetry for the local generalized conifold $G_{k,l}$ and orbifolded conifold $O_{k,l}$, following the construction of \cite{auroux07,CLL,AAK}.  For $k=l=1$, it reduces to mirror symmetry for the local conifold (which is self-mirror).  In this section we recall the geometries of $G_{k,l}$ and $O_{k,l}$ and their SYZ mirror symmetry.

\subsection{Lagrangian fibrations and SYZ mirrors}

\begin{defn}
A local generalized conifold is given by the equation 
$$xy=(1+z)^k(1+w)^l$$ 
in $\C^4$.  
\end{defn}

It is a toric Gorenstein singularity whose fan is the three-dimensional cone spanned by the primitive vectors $(0,0,1),(k,0,1),(0,1,1),(l,1,1)$.  In other words it is a cone over the trapezoid with vertices $(0,0,1),(k,0,1),(0,1,1),(l,1,1)$ contained in the affine plane in height 1.
When $k,l \geq 2$, the set of singularities is the union of $\{x=y=0,z=-1\}$ and $\{x=y=0,w=-1\}$; when $k=1$ and $l\geq 2$, the set of singularities is $\{x=y=0,w=-1\}$; when $k=l=1$, the point $x=y=1+z=1+w=0$ is the only singularity.  

A smoothing is given by changing the polynomial $(1+z)^k(1+w)^l$ (while keeping the highest order) such that its zero set contains no critical point.  A toric crepant resolution is given by subdividing the trapezoid into standard lattice triangles whose areas achieve the smallest value $1/2$.  The resolution is small in the sense that the exceptional locus is a rational curve.
For the purpose of SYZ mirror symmetry, we remove the anti-canonical divisor $\{zw=0\}$ and denote the resulting space as $G_{k,l}$.  Its smoothing and resolution are denoted as $\tilde{G}_{k,l}$ and $\hat{G}_{k,l}$ respectively.

\begin{defn}
A local orbifolded conifold is given by the equations 
$$u_1v_1=(1+z)^k, \, u_2v_2=(1+z)^l$$
in $\C^5$.  
\end{defn}

It is another toric Gorenstein singularity whose fan is the three-dimensional cone spanned by the primitive vectors $(0,0,1),(k,0,1),(0,l,1),(k,l,1)$.  In other words it is a cone over the corresponding rectangle in the affine plane in height 1.
When $k,l \geq 2$, the set of singularities is the union of $\{u_1=v_1=0,z=-1\}$ and $\{u_2=v_2=0,z=-1\}$; when $k=1$ and $l\geq 2$, the set of singularities is $\{u_2=v_2=0,z=-1\}$; when $k=l=1$, the point $u_1=v_1=u_2=v_2=1+z=0$ is the only singularity.  

A smoothing is given by changing the polynomials $(1+z)^k$ and $(1+z)^l$ such that they do not have multiple roots.  A toric crepant resolution is given by subdividing the rectangle into standard lattice triangles.  We remove the anti-canonical divisor $\{z=0\}$ and denote the resulting space by $O_{k,l}$.  Its smoothing and resolution are denoted as $\tilde{O}_{k,l}$ and $\hat{O}_{k,l}$ respectively.

In \cite{KL} we showed the following.

\begin{theorem}[\cite{KL}] \label{localSYZ}
The local generalized conifold $G_{k,l}$ is SYZ mirror to the orbifolded conifold $O_{k,l}$.  Namely, the deformed generalized conifold $\tilde{G}_{k,l}$ is SYZ mirror to the resolved orbifolded conifold $\hat{O}_{k,l}$; the resolved generalized conifold $\hat{G}_{k,l}$ is SYZ mirror to the deformed orbifolded conifold $\tilde{O}_{k,l}$.  It is summarized by the following diagram.
$$
\xymatrix{
\tilde{G}_{k,l}\ar@{<->}[d]_{SYZ} &\ar@{~>}[l] G_{k,l} \ar@{<->}[d]^{SYZ} & \ar@{->}[l]  \hat{G}_{k,l} \ar@{<->}[d]^{SYZ}  \\
\hat{O}_{k,l} \ar@{->}[r] & O_{k,l}\ar@{~>}[r] & \tilde{O}_{k,l}. 
}
$$
\end{theorem}

The SYZ program realizes a mirror pair as dual Lagrangian torus fibrations.  On $O_{k,l}$, we consider the Hamiltonian $T^2$-action given by $u_i \mapsto \lambda_i u_i$, $v_i \mapsto \lambda_i^{-1} v_i$ for $i=1,2$, leaving $z$ unchanged.  We also have the corresponding action on its resolution and smoothing, and let's denote the moment map to $\R^2$ by $\nu_{\mathbb{T}^2}$ (which is simply given by $(|u_1|^2 - |v_1|^2,|u_2|^2 - |v_2|^2)$ on $O_{k,l}$).  Then one can verify that
\begin{equation} \label{eq:fib_O}
(\log|z|,\nu_{\mathbb{T}^2})
\end{equation}
gives a Lagrangian fibration on $O_{k,l}$,$\hat{O}_{k,l}$ and $\tilde{O}_{k,l}$.

On $G_{k,l}$, we have the Hamiltonian $\bS^1$-action given by $x \mapsto \lambda x$, $y \mapsto \lambda^{-1} y$, leaving $z,w$ unchanged.  We also have the corresponding action on its resolution and smoothing, and let's denote the moment map to $\R$ by $\mu_{\mathbb{S}^1}$.  ($\mu_{\mathbb{S}^1}$ is simply given by $|x|^2 - |y|^2$ on $G_{k,l}$.)  Then we have the torus fibration
$$ (\mu_{\mathbb{S}^1},\log|z|,\log|w|) $$
on $G_{k,l}$,$\hat{G}_{k,l}$ and $\tilde{G}_{k,l}$.
The torus fibration is Lagrangian for $\hat{G}_{k,l}$, but not for $\tilde{G}_{k,l}$.
By the result of \cite{AAK} using the Moser argument, the torus fibration can be modified by an isotopy to a Lagrangian fibration (where the fibration map is piecewise smooth near the discriminant locus) with the first coordinate $\mu_{\bS^1}$ remains unchanged.  The issue is that the reduced symplectic space of $\tilde{G}_{k,l}$ by $\bS^1$ is just isotopic but not exactly equal to $(\C^2,\omega_{\mathrm{std}})$, and so one needs the Moser argument to connect these two by symplectomorphisms.  %{\color{red} Note that for the orbifold $G_{k,l}$ we have two hyperplanes intersecting at $(z,w)=(0,0)$ with multiplicities.  But this should not affect the Moser argument.}

From the Lagrangian fibrations on $\hat{G}_{k,l}$ and $\tilde{G}_{k,l}$, we constructed $\tilde{O}_{k,l}$ and $\hat{O}_{k,l}$ as their SYZ mirrors respectively; the reverse diection is also true.  The key ingredient in the construction is the holomorphic discs emanated from singular Lagrangian fibers (which are of Maslov index zero).  It is summarized as follows.

\begin{lemma}[Discriminant locus] \label{lem:disc-locus}
For the Lagrangian fibration on $G_{k,l}$ or the Lagrangian fibration on $O_{k,l}$, the discriminant locus is given by $\left(\{0\} \times \R \times \{0\}\right) \cup \left(\{0\}\times \{0\} \times \R\right)$.  On $\hat{G}_{k,l}$ or $\tilde{O}_{k,l}$, it becomes 
$$\left(\bigcup_{i=1}^k \left(\{s_i\} \times \R \times \{0\}\right)\right) \cup \left(\bigcup_{i=1}^l \left(\{t_i\} \times \{0\} \times \R \right)\right)$$
where $s_i$ and $t_i$ are certain constants related to the symplectic sizes of spheres in the exceptional curve of $\hat{G}_{k,l}$.  

For $\tilde{G}_{k,l}$ or $\hat{O}_{k,l}$, the discriminant locus is contained in the plane $\{0\} \times \R \times \R$ and is homotopic to the dual graph of the triangulation of the rectangle with vertices $(0,0),(k,0),(0,l),(k,l)$ in the definition of the resolution $\hat{O}_{k,l}$.
\end{lemma}

Denote by $B_0$ the complement of the discriminant locus in the base $B$ of a Lagrangian fibration with a Lagrangian section.  $B_0$ has an induced tropical affine structure (namely it has an atlas with transition maps being elements of $\mathrm{GL}(n,\Z) \ltimes \R^n$).  Denote by $\Lambda \subset TB_0$ the corresponding local system, and $\Lambda^* \subset T^*B_0$ the dual.  Then the inverse image of $B_0$ is symplectomorphic to $T^*B_0 / \Lambda^*$.  The complex manifold $TB_0 / \Lambda$ is called the semi-flat mirror which serves as the first-order approximation.  Then the generating functions of holomorphic discs of Maslov index zero give corrections to the complex structure of the semi-flat mirror.

\begin{prop}[Wall]
Given a Lagrangian fibration, Let $H \subset B$ (called the wall) be the set of regular Lagrangian torus fibers which bound non-constant holomorphic discs of Maslov index zero.  For the Lagrangian fibration on $G_{k,l}$ or $\hat{G}_{k,l}$, $H$ is given by $\left(\R \times \R \times \{0\}\right) \cup \left(\R \times \{0\} \times \R\right)$.  On $\tilde{G}_{k,l}$, $H$ is given by $\R \times \Delta$ where $\Delta$ denotes the discriminant locus contained in the horizontal plane.

For the Lagrangian fibration on $O_{k,l}$ or $\hat{O}_{k,l}$, $H$ is given by $\{0\} \times \R^2$.  On $\tilde{O}_{k,l}$, $H$ becomes $\{s_1,\ldots,s_k,t_1,\ldots,t_l\} \times \R^2$ where $s_i$ and $t_i$ are the constants appearing in Lemma \ref{lem:disc-locus}.
\end{prop}

The discriminant loci and walls are shown in Figure \ref{fig:local-walls-gencfd}.  

\begin{theorem}[Slab function]
Each component of $H$ is attached with a function defined by wall-crossing of the open Gromov-Witten potential.  For $\hat{G}_{k,l}$, the function attached to $\R \times \R \times \{0\}$ is given by $(1+Z)(1+q_1Z)\ldots(1+q_1\ldots q_{k-1}Z)$, and that attached to $\R \times \{0\} \times \R$ is given by $(1+cZ)(1+q_1'cZ)\ldots(1+q_1'\ldots q_{l-1}'cZ)$, where $q_1,\ldots,q_{k-1},q_1',\ldots,q_{l-1}',c$ are given in the form $e^{- \int_C \omega}$ for certain rational curves $C$, and $Z$ is the semi-flat complex coordinate corresponding to the third coordinate $\mu_{\bS^1}$ of the base of the Lagrangian fibration.  For $\tilde{G}_{k,l}$ the function is $1+Z$.

For $\tilde{O}_{k,l}$, the slab function attached to each $\{s_i\} \times \R^2$ is $(1+Z)$, and that attached to each $\{t_i\} \times \R^2$ is $(1+W)$, where $Z,W$ are the semi-flat complex coordinates corresponding to the base direction $\nu_{T^2}$ of the Lagrangian fibration.  For $\hat{O}_{k,l}$, the slab function is of the form 
$$\sum_{i=0}^k \sum_{j=0}^l (1+\delta_{ij}(q)) q^{C_{ij}}  Z^i W^j$$
where for each $i,j$, $\delta_{ij}(q)$ is a certain generating function of open Gromov-Witten invariants, and $C_{ij}$ is a certain rational curve in $\hat{O}_{k,l}$.
\end{theorem}

$\delta_{ij}(q)$ can be explicitly computed from the mirror map \cite{CCLT13}.  The discriminant loci, walls and slab functions are shown in Figure \ref{fig:local-walls-gencfd}.

\begin{figure}[h]
\begin{center}
\includegraphics[scale=0.5]{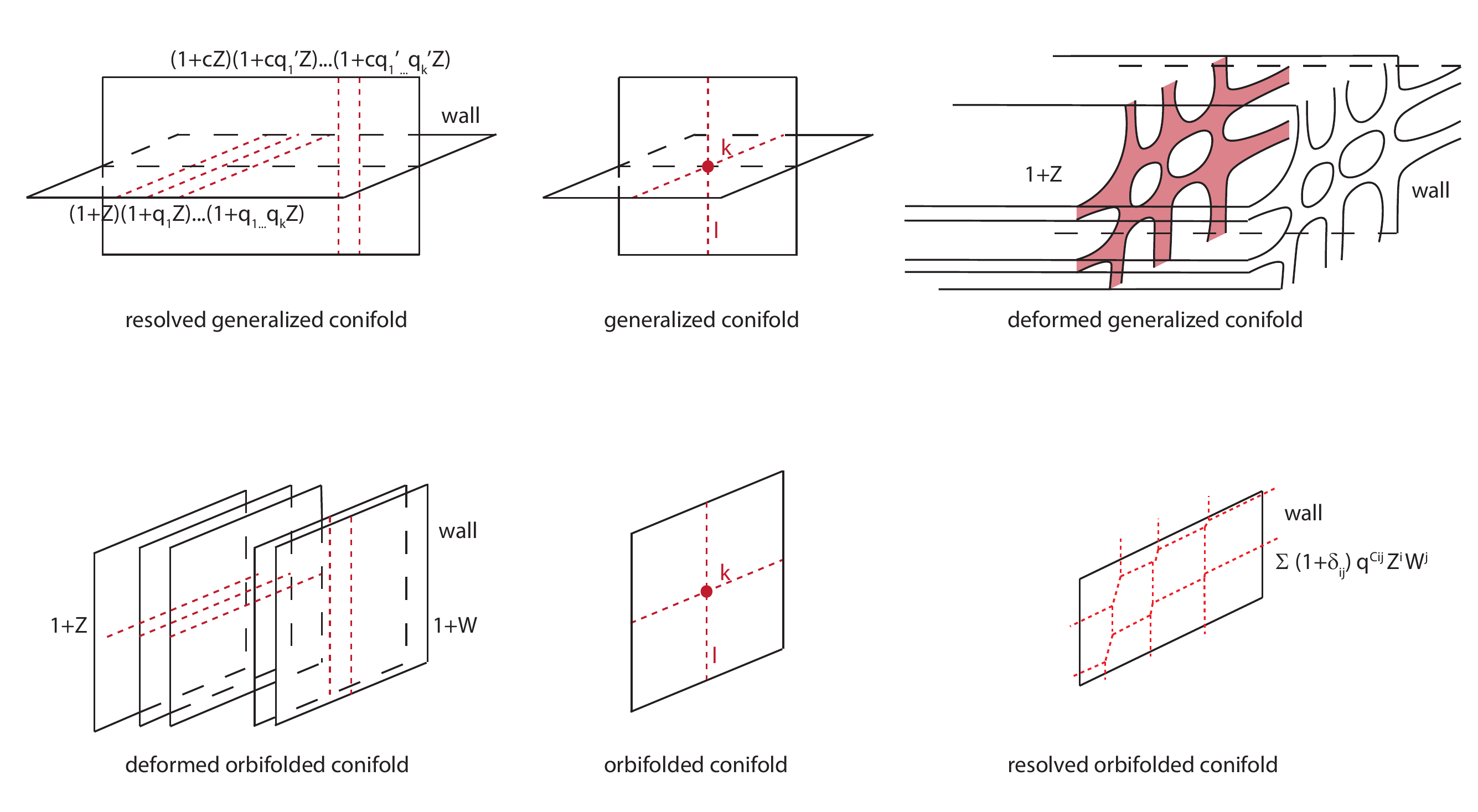}
\caption{The discriminant loci, walls and slab functions for the local generalized and orbifolded conifold transitions.}
\label{fig:local-walls-gencfd}
\end{center}
\end{figure}

By gluing local pieces of semi-flat mirrors using the slab functions, we obtain the SYZ mirrors given in Theorem \ref{localSYZ}.  We refer to \cite{CLL,KL} for detail.
In Section \ref{sec:loc-aff}, we shall follow \cite{GS07,CM1,CM2} and use tropical geometry to encode the data.  The Gross-Siebert program has the important advantage that it can handle global geometries by combinatorics.  The walls and generating functions of the local geometry given above serve as initial input data to the Gross-Siebert program.

\begin{remark}
In an ongoing work we shall construct the quiver mirror of $\tilde{O}_{k,l}$ by using the construction of \cite{CHL2}, see Figure \ref{fig:ncgencfd}, which is useful in studying stability conditions and flop along the line of \cite{FHLY}.  It is a noncommutative resolution of $G_{k,l}$ (which was extensively studied by \cite{Nagao,MN}, and can be derived from the construction of Bocklandt \cite{Bocklandt}). 
\end{remark}

\begin{figure}[h]
\begin{center}
\includegraphics[scale=0.5]{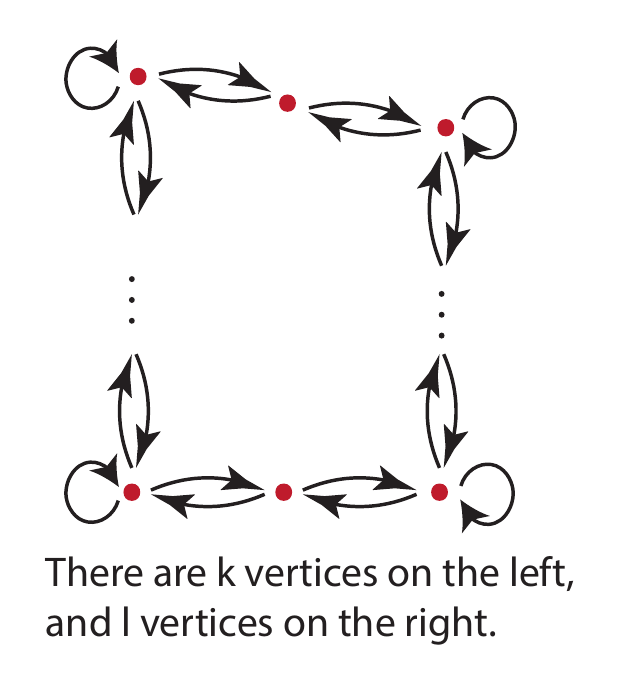}
\caption{A noncommutative resolution of the generalized conifold $G_{k,l}$, which also serves as a mirror of the orbifolded conifold.}
\label{fig:ncgencfd}
\end{center}
\end{figure}

In summary, we have SYZ mirror symmetry for the resolutions and smoothings of the local generalized and orbifolded conifold transitions near large volume limits.  There is also a noncommutative mirror construction near the conifold limit.  It is schematically depicted by Figure \ref{fig:gen-con-MS-mod}.

\begin{figure}[h]
\begin{center}
\includegraphics[scale=0.7]{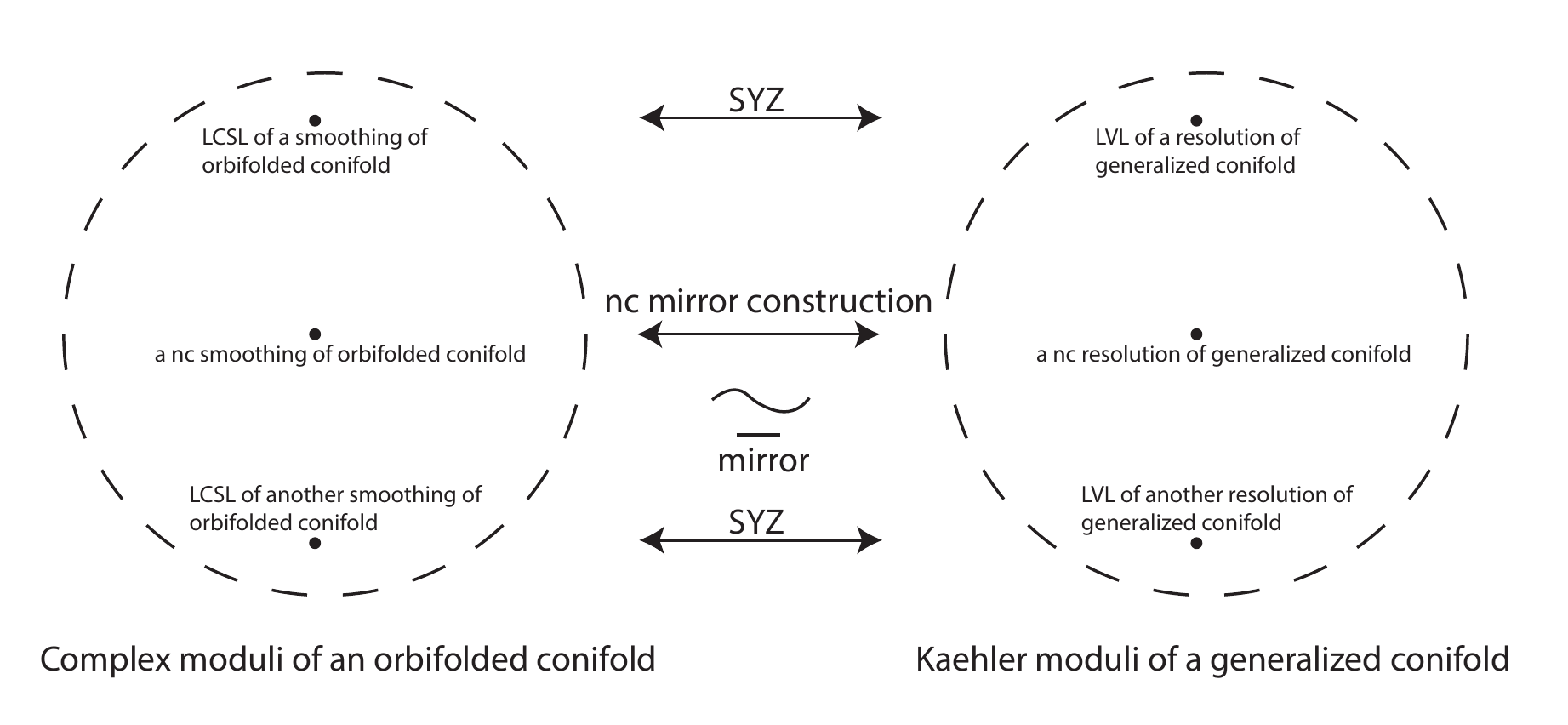}
\caption{The moduli spaces.  Around a large complex structure limit we have the SYZ construction using Lagrangian torus fibrations.  Around a conifold limit we have the noncommutative mirror construction in \cite{CHL2} using certain Lagrangian cycles.}
\label{fig:gen-con-MS-mod}
\end{center}
\end{figure}

\subsection{Monodromy computation for local generalized and orbifolded conifolds}
We have reviewed Lagrangian fibrations and SYZ for the local generalized conifold $G_{k,l}$ and orbifolded conifold $O_{k,l}$.  Now we compute the monodromies of the fibrations for later purpose.  

Recall that $B_0$ denotes the complement of discriminant locus in the base $B$ of the Lagrangian fibration.  For $G_{k,l}=\{xy=(1+z)^k(1+w)^l\}$, fix two contractible open sets $U_+$ and $U_-$ covering $B_0$, where $U_+ = B_0 - \R_{\leq 0} \times \Delta$ and $U_- = B_0 - \R_{\geq 0} \times \Delta$.  The torus bundles over $B_0$ trivialize over $U_+$ and $U_-$.  Then fix a basis of the fundamental group of each fiber at $(a,b,c) \in U_+$ as follows.  
%Fix the fiber $|z| = 1+\epsilon,|w|=1+\epsilon,\mu_{\bS^1}=0$ for $\epsilon>0$.
\begin{align*}
\gamma_1(t) &\textrm{ defined by } z=e^b,w=e^c,x= r e^{2\pi\consti t} \textrm{ for } r \in \R_+;\\
\gamma_2(t) &\textrm{ defined by } z=e^b e^{2\pi\consti t},w=e^c,x \in \R_+;\\ 
\gamma_3(t) &\textrm{ defined by } z=e^b,w=e^c e^{2\pi\consti t},x \in \R_+
\end{align*}
under the constraints $|x|^2-|y|^2=e^a$ and $xy=(1+z)^k(1+w)^l$).  (Note that $x\not=0$ over $U_+$.)
$[\gamma_i]$ for $i=1,2,3$ defines a basis of the fundamental group.  Similarly we fix a basis over $U_-$ (where $y\not=0$) by taking 
\begin{align*}
\tilde{\gamma}_1(t) &\textrm{ defined by } z=e^b,w=e^c,y= r e^{-2\pi\consti t} \textrm{ for } r \in \R_+;\\
\tilde{\gamma}_2(t) &\textrm{ defined by } z=e^b e^{2\pi\consti t},w=e^c,y \in \R_+;\\ 
\tilde{\gamma}_3(t) &\textrm{ defined by } z=e^b,w=e^c e^{2\pi\consti t},y \in \R_+.
\end{align*}
In the pre-image of $U_-$, $y \not=0$ and so the above is well-defined.

\begin{prop}[Monodromy for $G_{k,l}$] \label{prop:mon-gencfd}
For the Lagrangian fibration on $G_{k,l}$, the monodromy around $\{0\} \times \R \times \{0\}$ is $[\gamma_1] \mapsto [\gamma_1]$, $[\gamma_2] \mapsto [\gamma_2]$ and $[\gamma_3] \mapsto [\gamma_3] + l [\gamma_1]$; the monodromy around $\{0\} \times \{0\} \times \R$ is $[\gamma_1] \mapsto [\gamma_1]$, $[\gamma_2] \mapsto [\gamma_2]- k [\gamma_1]$ and $[\gamma_3] \mapsto [\gamma_3]$.
\end{prop}
\begin{proof}
First we consider $G_{k,l}$.  $U_+ \cap U_-$ consists of four connected components which we call chambers.  We have the following by considering the winding numbers of the variables $x,y,z,w$ with the constraint $xy=(1+z)^k(1+w)^l$.  For a base point in a chamber,
\begin{align*}
[\gamma_1]=[\tilde{\gamma}_1], [\gamma_2]=[\tilde{\gamma}_2],[\gamma_3]=[\tilde{\gamma}_3] \textrm{ in } \R \times \R_- \times \R_-;\\
[\gamma_1]=[\tilde{\gamma}_1], [\gamma_2]=[\tilde{\gamma}_2]-k[\tilde{\gamma}_1],[\gamma_3]=[\tilde{\gamma}_3] \textrm{ in } \R \times \R_+ \times \R_-;\\
[\gamma_1]=[\tilde{\gamma}_1], [\gamma_2]=[\tilde{\gamma}_2],[\gamma_3]=[\tilde{\gamma}_3]-l[\tilde{\gamma}_1] \textrm{ in } \R \times \R_- \times \R_+;\\
[\gamma_1]=[\tilde{\gamma}_1], [\gamma_2]=[\tilde{\gamma}_2]-k[\tilde{\gamma}_1],[\gamma_3]=[\tilde{\gamma}_3]-l[\tilde{\gamma}_1] \textrm{ in } \R \times \R_+ \times \R_+.
\end{align*} 

Now take a loop around $\{0\} \times \R \times \{0\}$.  It goes from the chamber $\R \times \R_+ \times \R_+$ to $\R \times \R_+ \times \R_-$ in $\R_+ \times \R \times \R$, and then goes back to the original chamber in $\R_- \times \R \times \R$.  In $\R_+ \times \R \times \R$ we use the basis $[\gamma_i]$; in $\R_- \times \R \times \R$ we use the basis $[\tilde{\gamma_i}]$ for $i=1,2,3$.  Then the monodromy around the loop is $[\gamma_1] \mapsto [\gamma_1]$, $[\gamma_2] \mapsto [\gamma_2]$ and $[\gamma_3] \mapsto [\gamma_3] + l [\gamma_1]$.  The computation for the monodromy around $\{0\} \times \{0\} \times \R$ is similar.
\end{proof}

For $O_{k,l}=\{u_1v_1=(1+z)^k,u_2v_2=(1+z)^l\}$, we take four contractible open subsets $U_{\pm\pm}$ covering $B_0$, where
\begin{align*}
U_{++} =& ((\R-\{0\})\times \R\times\R) \cup (\R \times \R_{>0} \times \R_{>0});\\
U_{+-} =& ((\R-\{0\})\times \R\times\R) \cup (\R \times \R_{>0} \times \R_{<0});\\
U_{-+} =& ((\R-\{0\})\times \R\times\R) \cup (\R \times \R_{<0} \times \R_{>0});\\
U_{--} =& ((\R-\{0\})\times \R\times\R) \cup (\R \times \R_{<0} \times \R_{<0}).
\end{align*}

We take a basis over $(a,b,c) \in U_{++}$ to be:
\begin{align*}
\gamma_1(t)=\gamma^{++}_1(t) &\textrm{ defined by } z=e^a e^{2\pi\consti t},u_1, u_2 \in \R_+;\\
\gamma_2(t)=\gamma^{++}_2(t) &\textrm{ defined by } z=e^a,u_1=r e^{2\pi\consti t} \textrm{ for } r \in \R_+,u_2 \in \R_+;\\ 
\gamma_3(t)=\gamma^{++}_3(t) &\textrm{ defined by } z=e^a,u_1\in \R_+,u_2=r e^{2\pi\consti t} \textrm{ for } r \in \R_+
\end{align*}
under the constraints $|u_1|^2-|v_1|^2=b$ and $|u_1|^2-|v_1|^2=c$.
It is well-defined since $u_1,u_2 \not= 0$ over $U_{++}$.
We replace $(u_1,u_2)$ by $(u_1,v_2^{-1})$ for $U_{+-}$, by $(v_1^{-1},u_2)$ for $U_{-+}$, and by $(v_1^{-1},v_2^{-1})$ for $U_{--}$.  Then we have a basis $\{\gamma^{\pm\pm}_i:i=1,2,3\}$ over each open set.

\begin{prop}[Monodromy for $O_{k,l}$] \label{prop:mon-orbcfd}
For the Lagrangian fibration on $O_{k,l}$, the monodromy around $\{0\} \times \R \times \{0\}$ is 
$[\gamma_1] \mapsto [\gamma_1] - l [\gamma_3]$, $[\gamma_2] \mapsto [\gamma_2]$ and $[\gamma_3] \mapsto [\gamma_3]$;
the monodromy around $\{0\} \times \{0\} \times \R$ is $[\gamma_1] \mapsto [\gamma_1] + k [\gamma_2]$, $[\gamma_2] \mapsto [\gamma_2]$ and $[\gamma_3] \mapsto [\gamma_3]$.
\end{prop}
\begin{proof}
The intersection of all the open sets is the union of the two disjoint open subsets $\R_\pm \times \R \times \R$.  In the chamber $\R_- \times \R \times \R$, for each $i$ we have $[\gamma_i^{\pm\pm}]$ to be all the same.  In the chamber $\R_+ \times \R \times \R$, for $i=2,3$, we have $[\gamma_i^{\pm\pm}]$ to be all the same; for $i=1$, $[\gamma_1^{++}]=[\gamma_1^{-+}]-k[\gamma_2^{-+}]=[\gamma_1^{+-}]-l[\gamma_3^{+-}]=[\gamma_1^{--}]-k[\gamma_2^{--}]-l[\gamma_3^{--}]$.  Now consider the monodromy around $\{0\} \times \R \times \{0\}$.  We start with the basis $[\gamma_i^{++}]$ in the chamber $\R_+ \times \R \times \R$, change to the basis $[\gamma_i^{+-}]$, move to the chamber $\R_- \times \R \times \R$ in $U_{+-}$, change back to the original basis, and move back to the original chamber in $U_{++}$.  Obviously $[\gamma_2^{++}],[\gamma_3^{++}]$ remains the same.  We have $[\gamma_1^{++}]\mapsto [\gamma_1^{++}] - l [\gamma_3^{++}]$.  The monodromy computation around $\{0\} \times \{0\} \times \R$ is similar.
\end{proof}

\section{Local affine geometries} \label{sec:loc-aff}

In this section, we formulate the local singularities, their resolutions and smoothings in the language of affine geometry.  Such a formulation has the great advantage that it can be easily globalized, thanks to the groundbreaking works of Gross-Siebert \cite{GS1,GS2,GS07}.

The notion of a tropical manifold is central in the Gross-Siebert program.  We briefly recall it below.

\begin{defn}
A polarized tropical manifold is a triple $(B,\cP,\phi)$, where $B$ is an integral affine manifold with singularities, $\cP$ is a toric polyhedral decomposition of $B$ (where the singularities occur in facets of polyhedrons in $\cP$),
and $\phi$ is a strictly convex multivalued piecewise linear function on $B$.  
\end{defn}

In their reconstruction program they also assume $(B,\cP,\phi)$ is positive and simple.  In dimension three it means that $\Delta$ is a trivalent graph and the vertices of $\Delta$ are of either positive or negative types (with certain simple monodromies of the affine connection around each edge), this is the case $d=1$ in Section \ref{sec:orb_vert}.  

Gross-Siebert defined discrete Legendre transform which associates a polarized tropical manifold $(B,\cP,\phi)$ to another one $(\check{B},\check{\cP},\check{\phi})$, with the property that the discrete Legendre transform of $(\check{B},\check{\cP},\check{\phi})$ goes back to $(B,\cP,\phi)$.  Their groundbreaking work constructs a toric degeneration of a Calabi-Yau manifold $X$ from each compact positive and simple polarized tropical manifold.  The Calabi-Yau manifolds associated to a polarized tropical manifold and its discrete Legendre transform form a mirror pair.

By gluing local models of Lagrangian fibrations with prescribed discriminant loci, Castano-Bernard and Matessi \cite{CM1} constructed a symplectic manifold (equipped with a Lagrangian fibration and a Lagrangian section) which contains $T^* B_0 / \Lambda^*$ as an open subset.  Here $B_0 = B - \Delta$ and $\Lambda^* \subset T^* B_0$ is a fiberwise lattice given by the affine structure.  In \cite{CM2}, they performed conifold transitions for the symplectic manifolds constructed from polarized tropical manifolds.

In this paper we follow the method of Castano-Bernard and Matessi \cite{CM1,CM2} to construct orbi-conifold transitions from an affine manifold with singularities.  First we construct the affine structures for the local generalized conifolds, orbifolded conifolds, and their resolutions and smoothings.  Then in Section \ref{sec:Schoen} we consider global affine structures with generalized or orbifolded conifold points.  The discriminant locus has four-valent vertices, and we also relax the simplicity condition (which allows orbifolded singularities).

Throughout this section $\{e_i:i=1,\ldots,n\}$ denotes the standard basis of $\R^n$.

\subsection{Local affine $A_{k-1}$ singularity in dimension two}
Let's begin with the local $A_{k-1}$ singularity in dimension two, where $k>1$.  

\begin{defn}[Affine $A_{k-1}$ singularity]
A singular point in an oriented affine surface is called an affine $A_{k-1}$ singularity (for $k>1$) if in a certain oriented basis, it has monodromy $\left(\begin{array}{cc}
1 & k \\
0 & 1
\end{array}\right)$.
A singular point with such monodromy (even when $k\leq 1$) is said to have multiplicity $k$.  When $k=\pm 1$ it is called simple.
\end{defn}

We can cook up two tropical manifolds with an affine $A_{k-1}$ singularity.  They are related to each other by discrete Legendre transform.  This reflects the fact that the $A_{k-1}$ singularity is self-mirror.

Define the tropical manifold $(B,\cP,\phi)$ as follows.  Take the lattice triangle 
$$\Conv\{(0,0),(0,k),(-1,k)\}$$ 
and the rectangle 
$$\Conv\{(0,0),(0,k),(1,0),(1,k)\}.$$ 
Glue them along the edge $\Conv\{(0,0),(0,k)\}$.  Then we have a manifold $B$ with corners and a polyhedral decomposition $\cP$.  See the top middle of Figure \ref{fig:affA_n}.  The fan structures at vertices are given as follows.  Denote the standard basis of $\R^2$ by $\{e_1,e_2\}$.  We have the fan generated by the primitive vectors $e_1,e_2,-e_1$.  At the vertex $(0,0)$, the tangent vectors $(1,0),(0,1),(-1,k)$ of the polytopes are mapped to the primitive generators $e_1,e_2,-e_1$ respectively.  At the vertex $(0,k)$, the tangent vectors $(-1,0),(0,-1),(1,0)$ are mapped to the primitive generators $e_1,e_2,-e_1$ respectively.  The multivalued piecewise linear function $\phi$ is defined by 
\begin{equation} \label{eq:phi1}
\phi(x,y) = \left\{ \begin{array}{ll}
x & x \geq 0 \\
0 & x \leq 0
\end{array}\right.
\end{equation}
in the fan generated by $\{e_1,e_2,-e_1\}$ at each of the two vertices.

\begin{figure}[h]
\begin{center}
\includegraphics[scale=0.7]{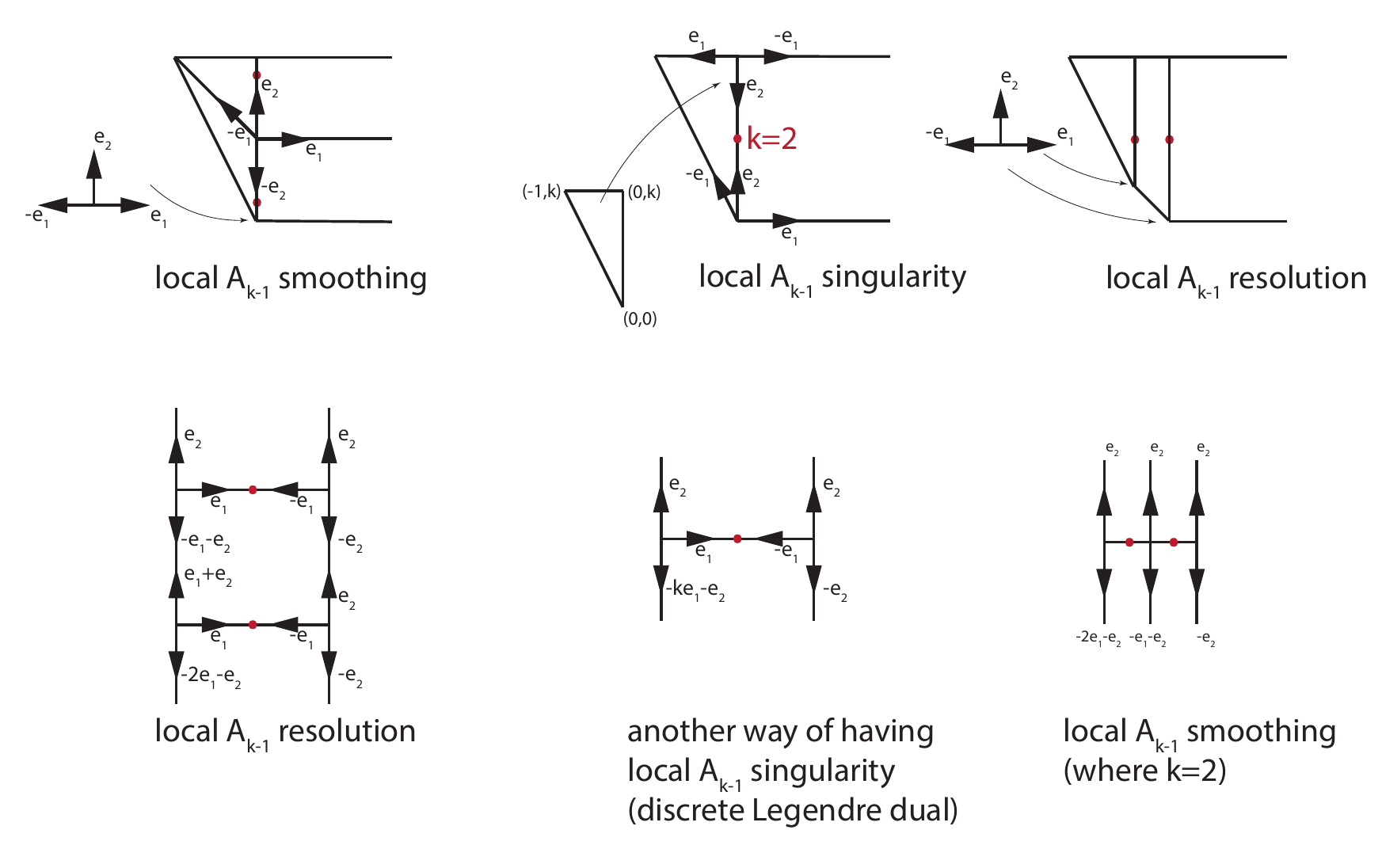}
\caption{Affine structures of the $A_{k-1}$ singularity, its resolution and smoothing.}
\label{fig:affA_n}
\end{center}
\end{figure}

The discriminant locus is a point $\Delta=\{p\}$ in the edge $\Conv\{(0,0),(0,k)\}$.  In other words $B-\Delta$ has an affine structure induced from the fans at the vertices.  The monodromy around $p$ is given as follows.  

\begin{prop}
For the tropical manifold $(B,\cP,\phi)$ given above, the monodromy matrix around $p$ in the standard basis $\{e_1,e_2\}$ at $(0,0)$ is
$\left(\begin{array}{cc}
1 & 0 \\
k & 1
\end{array}\right).$
\end{prop}
\begin{proof}
Consider a loop starting from the vertex $(0,0)$, going to the vertex $(0,k)$ in the rectangle, and going back to the vertex $(0,0)$ in the triangle.  It is easy to see that the vector $e_2$ of the fan at the vertex $(0,0)$ is monodromy invariant.  Consider the vector $e_1$.  Transporting it to the vertex $(0,k)$ in the rectangle, it is identified with $-e_1$ in the fan at the vertex $(0,k)$.  Transporting it back to the vertex $(0,0)$ in the triangle, it is identified with $e_1+ke_2$ in the fan at the vertex $(0,0)$ (since $(1,0) = -(-1,k)+k(0,1)$).  Thus $e_1$ is sent to $e_1+ke_2$ under monodromy.  
\end{proof}

When $k=1$ the singularity is removable, and the affine manifold is simple in the sense of Definition 1.60 of \cite{GS1}.  However the monodromy is no longer simple for $k>1$.

Next, we define $(\check{B},\check{\cP})$ by gluing two squares
$$\Conv\{(0,0),(1,0),(0,1),(1,1)\} \textrm{ and } \Conv\{(0,0),(1,0),(0,-1),(1,-1)\}$$ 
along the edge $\Conv\{(0,0),(1,0)\}$.  See the bottom middle of Figure \ref{fig:affA_n}.  The fan at the vertex $(0,0)$ is given by mapping the tangent vectors $(0,-1),(1,0),(0,1)$ to $-ke_1-e_2, e_1, e_2$ respectively.  The fan at the vertex $(1,0)$ is given by mapping the tangent vectors $(0,1),(-1,0),(0,-1)$ to $e_2,-e_1,-e_2$ respectively.  The discriminant locus is a point $\Delta=\{p\}$ in the edge $\Conv\{(0,0),(1,0)\}$.  One can similarly check that the monodromy is given as follows and the proof is omitted.

\begin{prop}
For $(\check{B},\check{\cP})$, the monodromy matrix around $p$ in the standard basis $\{e_1,e_2\}$ at $(0,0)$ is
$\left(\begin{array}{cc}
1 & -k \\
0 & 1
\end{array}\right).$
\end{prop}

\begin{remark}
This is just equivalent to the affine structure of $(B,P)$ if we switch to the basis $(-e_2,e_1)$.
\end{remark}

Define the multivalued piecewise linear function $\check{\phi}$ by
$$
\check{\phi}(x,y) = \left\{ \begin{array}{ll}
0 & \textrm{ in } \R_{\geq 0}\{e_1,-k e_1 - e_2\} \\
ky & \textrm{ in } \R_{\geq 0}\{e_1,e_2\} \\
\end{array}\right.
$$
on the fan at the vertex $(0,0)$, and
$$
\phi(x,y) = \left\{ \begin{array}{ll}
0 & \textrm{ in } \R_{\geq 0}\{-e_1,-e_2\} \\
ky & \textrm{ in } \R_{\geq 0}\{-e_1,e_2\} \\
\end{array}\right.
$$
on the fan at the vertex $(1,0)$.  (The differentials, or so called discrete Legendre transform, give the corners of the polytopes of $(B,P)$.)

It is easy to check the following.

\begin{prop}
$(B,\cP,\phi)$ and $(\check{B},\check{\cP},\check{\phi})$ given above are discrete Legendre dual to each other.
\end{prop}
\begin{proof}
One can directly check that the polytopes in $(\check{B},\check{\cP})$ are Legendre dual polytopes of the piecewise linear function $\phi$ around vertices of $(B,\cP)$ (by taking the differential of $\phi$ restricted on each cone), and vice versa.  Moreover the fan structure at each vertex of $(\check{B},\check{\cP})$ is given by the normal fan of the corresponding polytope in $(B,\cP)$, and vice versa.
\end{proof}

\begin{remark} \label{rem:edge}
By taking a product of the affine $A_{k-1}$ singularity $(B,\cP)$ (or $(\check{B},\check{\cP})$) with the affine line $\R$, one obtains a tropical threefold with discriminant locus being a line.  The multiplicity of such a discriminant locus is defined to be $k$.
(As in \cite{CM1,CM2}, one can perturb the discriminant locus to be a curve.)
\end{remark}

\subsection{Smoothing and resolution of a local affine $A_{k-1}$ singularity}
The tropical manifolds corresponding to the resolution and smoothing of an $A_{k-1}$ singularity are given by the left and right sides of Figure \ref{fig:affA_n}.  The singularity in the affine base (which has multiplicity $k$) separates into $k$ simple singularities.  In a smoothing the simple singularities lie in the same monodromy-invariant affine hyperplane (which is formed by some edges of the polytopes in the decomposition).  A resolution is Legendre dual to a smoothing.

The readers can easily understand the polyhedral decompositions and fan structures from the figures and we omit the detailed descriptions.  The monodromy around each critical point is simple, namely it equals 
$\left(\begin{array}{cc}
1 & \pm 1 \\
0 & 1
\end{array}\right)$
up to conjugation.

For the fan generated by $\{e_1,e_2,-e_1\}$, the restriction of the multivalued piecewise linear function is given by Equation \eqref{eq:phi1}; for that generated by $\{e_1,e_2,-e_1,-e_2\}$ in the top left of Figure \ref{fig:affA_n}, it is given by
$$
\phi(x,y) = \left\{ \begin{array}{ll}
x+y & x, y \geq 0 \\
x & x \geq 0 \textrm{ and } y \leq 0 \\
y & y \geq 0 \textrm{ and } x \leq 0 \\
0 & x, y \leq 0;
\end{array}\right.
$$
for that generated by $\{e_1,e_2,-e_1,-j e_1 - e_2\}$ in the bottom right of Figure \ref{fig:affA_n}, it is given by
$$
\phi(x,y) = \left\{ \begin{array}{ll}
0 & \textrm{ in } \R_{\geq 0}\{e_1,-j e_1 - e_2\} \\
\frac{(j+k)(j+k-1)y}{2} & \textrm{ in } \R_{\geq 0}\{e_1,e_2\} \\
-x+\frac{(j+k)(j+k-1)y}{2} & \textrm{ in } \R_{\geq 0}\{e_2,-e_1\} \\
-x+j y & \textrm{ in } \R_{\geq 0}\{-e_1,-j e_1 - e_2\} \\
\end{array}\right.
$$
$\phi$ restricted to other strata are similar.

\begin{remark} \label{rem:subdivide}
At the vertex $(-1,k)$ of $(B,\cP)$ for the local $A_{k-1}$ singularity, we have an orbifolded fan structure (see the top middle of Figure \ref{fig:affA_n}).  The same happens for its local $A_{k-1}$ resolution (top right of Figure \ref{fig:affA_n}).  It can be easily resolved by a subdivision of the polyhedral decomposition.  See Figure \ref{fig:affA_nres}.  The fan structures at the additional vertices are taken to be trivial.
\end{remark}

\begin{figure}[h]
\begin{center}
\includegraphics[scale=0.6]{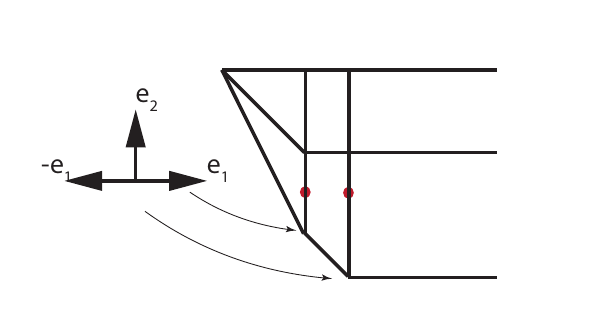}
\caption{A subdivision to resolve the orbifold fan.  It shows the case for $k=2$, and it is similar for general $k$.}
\label{fig:affA_nres}
\end{center}
\end{figure}

\begin{remark} \label{rem:edge-sm}
Again by taking a product with $\R$, one obtains a smoothing or a resolution of the singularity given in Remark \ref{rem:edge}.
\end{remark}

\subsection{Lagrangian fibration on local $A_{k-1}$ singularity}
The $A_{k-1}$ surface singularity is toric.  Its fan is the cone in $\R^2$ generated by $(0,1)$ and $(k,1)$.  Let $X$ be the corresponding toric variety.  Picking up a toric K\"ahler form, one has the moment map to $\R^2$ whose image is a (non-compact) polyhedral set.  Let $\mu_1$ be the horizontal component of the moment map.  Then 
\begin{equation} \label{eq:fib_An}
(\mu_1,\log|\nu-1|)
\end{equation}
gives a Lagrangian fibration on $X - \{\nu=1\}$ where $\nu$ is the toric holomorphic function corresponding to the $(0,1)$ lattice point.  The affine base of the fibration is isomorphic to the top middle of Figure \ref{fig:affA_n}.  See \cite{LLW} for more detail.

Alternatively we can take its mirror variety $\{(u,v,z)\in \C^2 \times \C^\times:uv=(1+z)^k\}$ which is again an $A_{k-1}$ singularity at $z=-1,u=v=0$.  
\begin{equation} \label{eq:fib_An'}
(|u|^2-|v|^2,\log|z|)
\end{equation}
gives a Lagrangian fibration whose affine base is isomorphic to the bottom middle of Figure \ref{fig:affA_n}.  

The affine smoothings and resolutions in Figure \ref{fig:affA_n} are base of Lagrangian fibrations on resolution and smoothing of the $A_{k-1}$ singularity.  We have the toric resolution whose fan is generated by $(j,1)$ for $j=0,\ldots,k$.  A smoothing is given by deforming the right hand side of the equation $uv=(1+z)^k$ such that all roots are simple.  Their Lagrangian fibrations are given by the same equations as above.

By taking product with $\C^\times$, we have Lagrangian fibrations whose base are isomorphic to the affine threefolds in Remark \ref{rem:edge} and \ref{rem:edge-sm}.  These Lagrangian fibrations serve as local models which can be glued to give more interesting geometries.

\subsection{Affine models for local generalized and orbifolded conifolds} \label{sec:aff-loc}
Now we construct tropical manifolds modeling the base of the Lagrangian fibrations in Section \ref{sec:SYZ}.  They will have the same discriminant loci and monodromies as in the last subsection.  The singularities of the local generalized and the orbifolded conifolds are said to be negative and positive respectively, according to the Euler characteristics of the corresponding singular Lagrangian fibers.  For $k=l=1$, this is the local affine geometry studied by \cite{CM2}.

First we consider the local generalized conifold.  Take the triangular prisms 
$$\Conv\{(0,0,0),(l,0,0),(0,0,k),(l,0,k),(0,-1,k),(l,-1,k)\}$$ 
and 
$$\Conv\{(0,0,0),(l,0,0),(0,0,k),(l,0,k),(0,1,0),(0,1,k)\}$$
and glue them together along the rectangle $\Conv\{(0,0,0),(l,0,0),(0,0,k),(l,0,k)\}$.  See the left of Figure \ref{fig:loc-orb-cfd}.

Take the fan given by the product of $\R_{\geq 0}\{e_1\}$ with the fan generated by $\{e_2,e_3,-e_2\}$ in the plane.  Then the fan at the vertex $(0,0,0)$ is given by mapping the tangent vectors $(1,0,0),(0,1,k),(0,0,1),(0,-1,0)$ to the generators $e_1,e_2,e_3,-e_2$ respectively.  The fans at the other three vertices $(l,0,0),(0,0,k),(l,0,k)$ are similar and the reader can understand from the left of Figure \ref{fig:loc-orb-cfd}.

\begin{figure}[h]
\begin{center}
\includegraphics[scale=0.7]{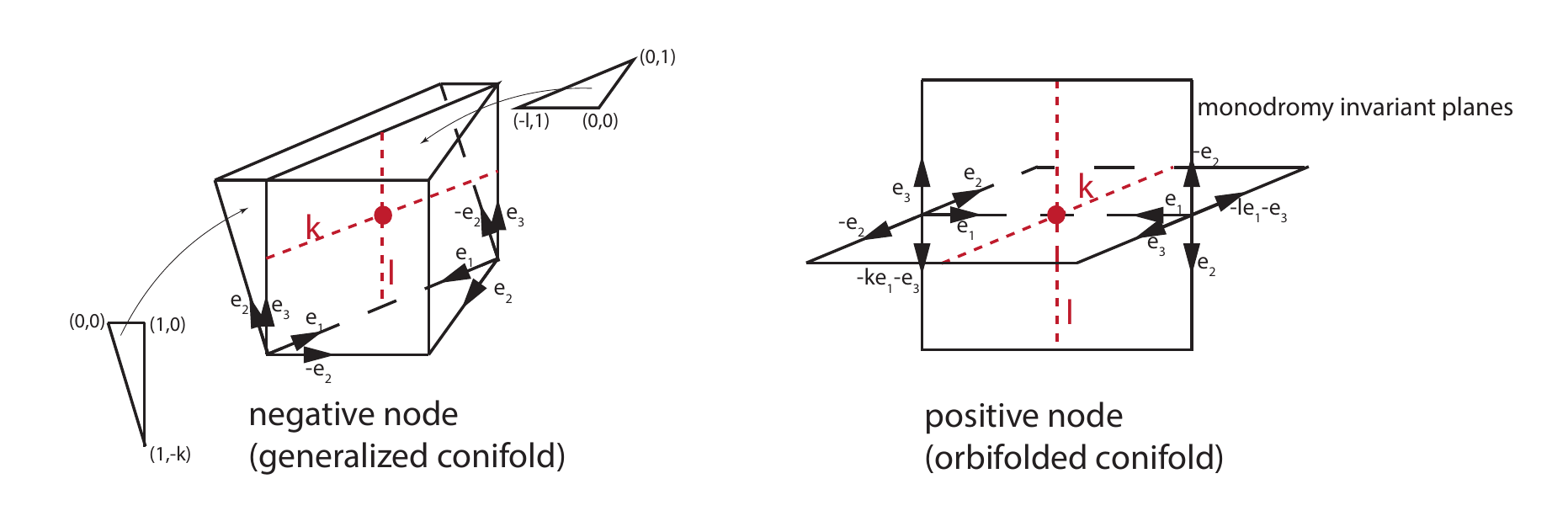}
\caption{Affine structures of the local generalized conifold and orbifolded conifold.}
\label{fig:loc-orb-cfd}
\end{center}
\end{figure}

The discriminant locus $\Delta$ is a union of the lines 
$$\left([0,l] \times \{0\} \times \{k/2\}\right) \cup \left(\{l/2\} \times \{0\} \times [0,k] \right)$$ 
in the rectangle $\Conv\{(0,0,0),(l,0,0),(0,0,k),(l,0,k)\}$.  This gives an affine manifold with singularities and polyhedral decomposition $(B,\cP)$.  The monodromy is given as follows.

\begin{prop} \label{prop:mono-match}
For the tropical manifold $(B,\cP)$ given above, the monodromy matrix around the component $[0,l] \times \{0\} \times \{k/2\}$ of the discriminant locus in the standard basis $\{e_1,e_2,e_3\}$ at $(0,0)$ equals to
$$\left(\begin{array}{ccc}
1 & 0 & 0\\
0 & 1 & 0\\
0 & k & 1\\
\end{array}\right)$$
and the monodromy matrix around the component $\{l/2\} \times \{0\} \times [0,k]$ equals to
$$\left(\begin{array}{ccc}
1 & -l & 0\\
0 & 1 & 0\\
0 & 0 & 1\\
\end{array}\right).$$
It matches with the monodromy in Proposition \ref{prop:mon-orbcfd}.
\end{prop}
\begin{proof}
$e_1$ and $e_3$ are obviously monodromy invariant.  Let's consider the monodromy of $e_2$ around $\{l/2\} \times \{0\} \times [0,k]$, and the other case is similar.  First we transport $e_2$ from the vertex $(0,0,0)$ to $(l,0,0)$ in the prism on the right, which is identified with $l e_1-e_2$ in the fan at $(l,0,0)$ (since $(0,1,0)=l(-1,0,0)-(-l,-1,0)$).  It is identified with the vector $(-l,0,0)+(0,1,k)$ in the prism on the left.  Transporting it back to $(0,0,0)$, it is identified with $-l e_1 + e_2$ in the fan at $(0,0,0)$.  Thus the monodromy maps $e_2$ to $-l e_1 + e_2$.
\end{proof}

The multivalued piecewise linear function $\phi$ is defined by
\begin{equation} \label{eq:phi-gc}
\phi(x,y,z) = \left\{ \begin{array}{ll}
y & y \geq 0 \\
0 & y \leq 0 \\
\end{array}\right.
\end{equation}
on the fan generated by $e_1,e_2,e_3,-e_2$.  $(B,\cP,\phi)$ defines a tropical manifold.

The discrete Legendre transform of $(B,\cP,\phi)$ is a union of four rectangles as shown in the right of Figure \ref{fig:loc-orb-cfd}.  The readers can work out the multivalued piecewise linear function $\check{\phi}$ from the figure, and check that the monodromy matches with that in Proposition \ref{prop:mon-gencfd}.

\begin{remark}
Proposition 6.4 of \cite{CM2} extends to this case by the same proof using action coordinates.  Namely, the affine structure induced on the base of the Lagrangian fibration \eqref{eq:fib_O} on $O_{k,l}$ is isomorphic to the affine manifold $(\check{B},\check{\cP})$.

The monodromy of $(B,\cP)$ matches with that of the Lagrangian fibration on the orbifolded conifold, while the monodromy of $(\check{B},\check{\cP})$ matches with that of the Lagrangian fibration on the generalized conifold.  It may look confusing that the roles of $(B,P)$ and $(\check{B},\check{P})$ get switched.  This is due to the fact that $T^*B_0 \cong T\check{B}_0$ and $T^*\check{B}_0 \cong TB_0$.  (Away from singular fibers a Lagrangian fibration is modeled by $T^*B_0 / \Lambda^*$.)  $TB_0/\Lambda$ and $T^*B_0 / \Lambda^*$ are related by discrete Legendre transform.  
\end{remark}

\subsection{Smoothings and resolutions of generalized and orbifolded conifolds}

The affine generalized conifolds are not simple in the sense of \cite[Definition 1.60]{GS1}.  They can be smoothed or resolved as shown by the left and right sides of Figure \ref{fig:affA_n} (for the case $k=l=2$).  Smoothing and resolution of the orbifolded conifold are given by the Legendre transform.  There are different choices of smoothing corresponding to different ways of refining the rectangle $[0,k]\times [0,l]$ into standard triangles; the discriminant locus is given by taking the dual graph of the triangulation.  (There are further choices of extending the triangulation to a refinement of the polyhedral decomposition, but it does not affect the affine structure.)  Similarly there are different choices of resolution corresponding to the different orders of horizontal and vertical line components of the discriminant locus.  In the description below we have fixed a particular choice.

\begin{figure}[h]
\begin{center}
\includegraphics[scale=0.7]{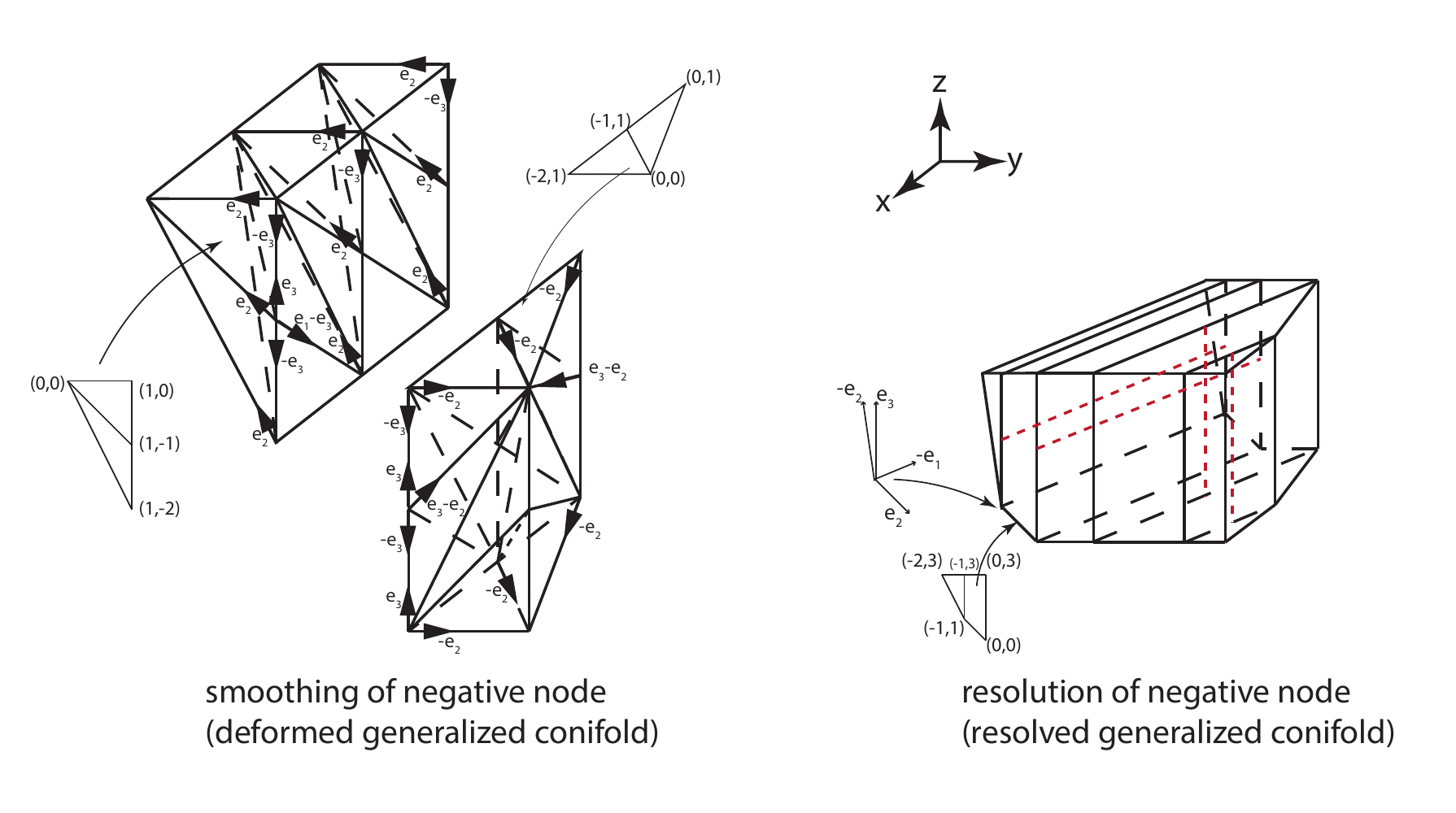}
\caption{Affine base of the smoothing and resolution of the generalized conifold.  In the figure $k=l=2$.}
\label{fig:loc-orb-con-res}
\end{center}
\end{figure}

For the smoothed generalized conifold, the polyhedral decomposition is shown in the left figure.  The non-trivial affine structure comes from the assignment of a fan structure at each lattice point in the rectangle in the direction $(1,0,0),(0,0,1)$.  Let's take the coordinate system such that these lattice points are given by $(j,0,i)$ for $i=0,\ldots,k$ and $j=0,\ldots,l$.  The key point is that, at the lattice point $(j,0,i)$, the vector $e_2$ in the local chart induced by the fan is identified with $(0,-1,k-i)$ in the polyhedral decomposition, and the vector $-e_2$ is identified with $(l-j,1,0)$.  Other directions are trivially identified.  

The triangulation of the rectangle $[0,k]\times [0,l]$ is associated with a dual graph.  To fix the positions of the vertices in the dual graph, we fix an integral piecewise linear function supported on the triangulation.  Then we define $\phi$ to be the sum of this function and that given by Equation \eqref{eq:phi-gc}.  This completes the definition of the tropical manifold $(B,P,\phi)$ corresponding to a deformed generalized conifold.  Its Legendre dual corresponds to resolution of a resolved orbifolded conifold.

The resolved affine generalized conifold is shown in the right figure.  The discriminant locus consists of vertical lines contained in the planes $y=1,\ldots,l$ and horizontal lines contained in the planes $y=0,\ldots,-k+1$.  Let's fix the coordinates such that the polyhedral decomposition contains the vertices $((j-1)j/2,j,0)$, $((j-1)j/2,j,k(k+1)/2)$ for $j=1,\ldots,l$, and $(0,-i+1,(i-1)i/2)$ or $(l(l+1)/2,-i+1,(i-1)i/2)$ for $i=1,\ldots,k$.  At the vertex $((j-1)j/2,j,0)$ or $((j-1)j/2,j,k(k+1)/2)$, $e_2$ is identified with $(j,1,0)$ and $-e_2$ is identified with $(-j+1,-1,0)$.  At the vertex $(0,-i+1,(i-1)i/2)$ or $(l(l+1)/2,-i+1,(i-1)i/2)$, $e_2$ is identified with $(0,1,-i+1)$ and $-e_2$ is identified with $(0,-1,i)$.

The multivalued piecewise linear function $\phi$ is defined by asserting that in the fan of each vertex, $\phi(e_2)=1$ and $\phi(-e_2)=\phi(\pm e_3)=\phi(\pm e_1)=0$.  This defines the tropical manifold $(B,P,\phi)$ corresponding to a resolved generalized conifold, whose Legendre dual corresponds to a deformed orbifolded conifold.

One can directly verify the following.

\begin{prop}
The tropical manifolds given above are positive and simple in the sense of \cite{GS1}.
\end{prop}

\subsection{Orbifolded trivalent vertex} \label{sec:orb_vert}

In this subsection we introduce an orbifolded version of the trivalent vertex in the Gross-Siebert program.  This will be used in Definition \ref{def:orbi-cfd}.  See \cite[Example 2.2 and 2.3]{CM2} for the usual trivalent vertex.

Take a lattice triangle $T \subset \R^2$ with vertices $v_0, v_1, v_2 \in \Z^2$ (labeled clockwise, and the indices are taken in $\Z/3\Z$).  We may simply take $v_0=0$.  It gives orbifolded positive and negative vertices as follows.  The negative vertex consists of the prism $T \times [0,1] \subset \R^3$ and the simplex 
$$\mathrm{Conv}\{(v_0,0),(v_1,0),(v_2,0),(v_0,-1)\} \subset \R^3.$$  
They are glued along the face $T \times \{0\}$ as shown in Figure \ref{fig:orb_vert}.  The fan structure at $(v_0,0)$ is generated by $(u_1,0),(-u_3,0),e_3,-e_3$ (which can be orbifolded) where $u_i$ are the primitive vectors in the directions of $v_i-v_{i-1}$, and they are mapped to the tangent vectors of the polytopes at $(v_0,0)$ trivially.  The fan structure at $(v_1,0)$ is generated by $(u_2,0),(-u_1,0),e_3,-e_3$ where $(u_2,0),(-u_1,0)$ are mapped to the tangent vectors of the polytopes trivially, while $e_3$ is mapped to $(0,0,1)$ and $-e_3$ is mapped to $(v_0-v_1,-1)$.  Similarly the fan structure at $(v_2,0)$ is generated by $(u_2,0),(-u_1,0),e_3,-e_3$ where $(u_2,0),(-u_1,0)$ are mapped to the tangent vectors of the polytopes trivially, $e_3$ is mapped to $(0,0,1)$ and $-e_3$ is mapped to $(v_0-v_2,-1)$.

\begin{figure}[h]
\begin{center}
\includegraphics[scale=0.6]{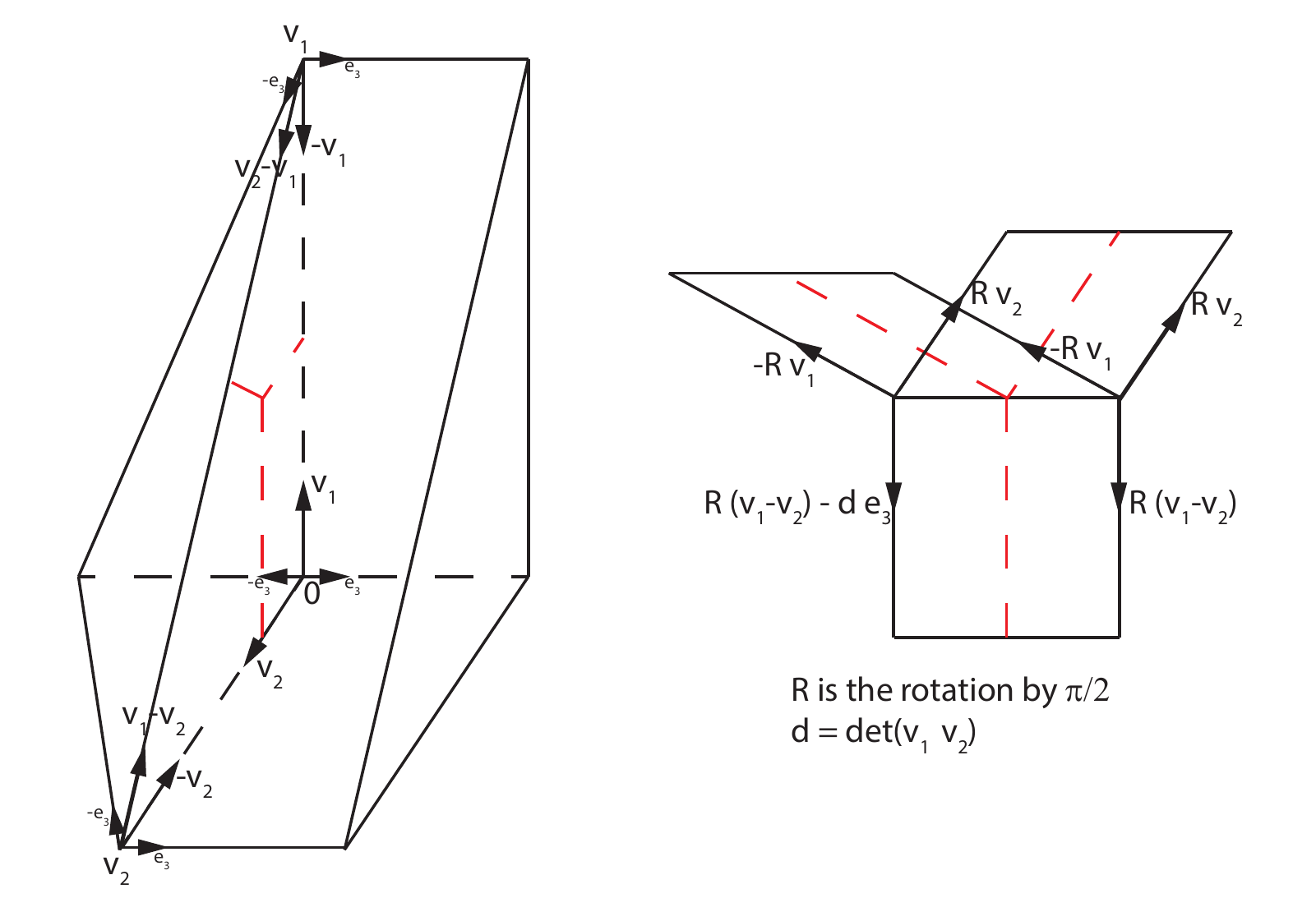}
\caption{The negative (the left) and positive (the right) orbifolded trivalent vertices.}
\label{fig:orb_vert}
\end{center}
\end{figure}

The discriminant locus is a Y-shape which is the dual graph of $T$ embedded in $T\times\{0\}$.  It is easy to verify the following.

\begin{prop}
$\R^2\times\{0\}$ is monodromy invariant.  The monodromy around the leg dual to the edge $\{v_{i-1},v_i\}$ of $T$ sends $e_3$ to $e_3 + (v_i-v_{i-1})$.  ($v_3:=v_0$.)  In particular the multiplicity of the legs is equal to the affine lengths of $v_i-v_{i-1}$.
\end{prop}

The piecewise linear function $\phi$ is given by Equation \eqref{eq:phi-gc}.  This finishes the definition of an orbifolded negative vertex as a tropical manifold.

The orbifolded positive vertex associated to $T$ is the following.  Let 
$$R = \left( 
\begin{array}{ll} 0 & -1\\
1 & 0 
\end{array}
\right).$$
$-R\cdot (v_i-v_{i-1})$ are outward normal vectors of the triangle $T$.  The polyhedral decomposition consists of the three polyhedral sets $\left(\R_{\geq 0} \cdot \{-R\cdot (v_i-v_{i-1}),-R\cdot (v_{i+1}-v_i)\}\right) \times [0,1]$ for $i=1,2,3$, and they are glued as shown on the right of Figure \ref{fig:orb_vert}.  The fan at $(0,1)$ is generated by the vectors $-R \cdot u_1,-R \cdot u_2, -R \cdot u_3, -e_3$.  The fan at $(0,0)$ is generated by the directions $-R \cdot u_1,-R \cdot u_3, -R \cdot (v_2-v_1) - d e_3, e_3$ for $d = \det (v_1, v_2)$ (two times the area of $T$), where $-R \cdot u_1,-R \cdot u_3, e_3$ are mapped trivially to the tangent vectors of the polytopes and $-R \cdot (v_2-v_1) - d e_3$ is mapped to $-R \cdot (v_2-v_1)$. The discriminant locus is the same Y-shape which is union of the intersection of the facets of the polytopes with $\R^2 \times \{0\}$.

The multivalued piecewise linear function restricts to the fan at each of the two vertices as $v_0,v_1,v_2$ (regarded as linear function on the dual vector space) on the three maximal cones of the fan respectively.  This gives an orbifolded positive vertex as a tropical manifold.

\begin{prop}
Take the leg in the direction $-R \cdot u_i$.  The plane $\R \cdot \{(-R \cdot (v_i-v_{i-1}),0), (0,1)\}$ is monodromy invariant.  Let $w \in \Z^2\times\{0\}$ such that $\det (u_i,w) =1$.  Then the monodromy sends $-R \cdot w$ to $-R \cdot w - h e_3$ where $h$ is the affine length of $v_i-v_{i-1}$.  In particular the multiplicity is equal to the affine length of $v_i-v_{i-1}$.

The orbifolded positive and negative vertices defined above form a Legendre dual pair.
\end{prop} 

\begin{proof}
It is obvious that the plane $\R \cdot \{(-R \cdot (v_i-v_{i-1}),0), (0,1)\}$ is monodromy invariant, and the monodromy sends $(-R \cdot (v_{i+1}-v_i),0)$ to $(-R \cdot (v_{i+1}-v_i),-d)$.  Write $v_{i+1}-v_i = -a u_i + b w$ for some integers $a,b$.  Also $v_i-v_{i-1}= h u_i$.  We have $d = \det (v_{i+1}-v_i,-(v_i-v_{i-1})) = bh$.  Then the monodromy sends $-R \cdot w = -R\cdot (v_{i+1}-v_i)/b - a R \cdot u_i / b$ to $-R\cdot (v_{i+1}-v_i)/b - d e_3/b - a R \cdot u_i / b = -R \cdot w - he_3$.  The remaining statements are easy to check.
\end{proof}

The orbifolded positive vertex corresponds to the toric CY orbifold $\C^3/G$ for a finite group $G$, whose fan is the cone over $T \times \{1\} \subset \R^3$.  Its Lagrangian fibration is again given by Equation \eqref{eq:fib_An}, where in this case $\mu_1$ is the first two components of the moment map (with respect to a fixed toric K\"ahler form).  The orbifolded negative vertex corresponds to the mirror variety $\{(u,v,z_1,z_2) \in \C^2 \times (\C^\times)^2: uv=\sum_{i=1}^3 z_1^{a^{(i)}_1}z_2^{a^{(i)}_2}\}$ where $v_i=(a^{(i)}_1,a^{(i)}_2)$ are the vertices of the triangle $T$.  One can cook up a piecewise smooth Lagrangian fibration by using the construction of \cite{AAK} by Moser argument.

\subsection{Local Gorenstein singularity} \label{sec:Gor}

A natural further generalization is the tropical manifolds corresponding to toric Gorenstein singularity and its mirror.  Orbifolded trivalent vertices and conifolds are included as special cases.  SYZ mirror symmetry for geometric transitions associated to toric Gorenstein singularities was studied in \cite{L13}.

The construction is very similar to the last subsection, and so we will not go into detail.  The lattice triangle $T$ in the last subsection is replaced by a lattice polygon $P \subset \R^2$ with vertices $v_0,\ldots,v_{m-1}$.  Then the discriminant locus consists of legs in directions $-R\cdot (v_i-v_{i-1})$ stemming from a single vertex.  The integer $d$ is defined to be two times the area of $P$ here.  We obtain the Gorenstein positive and negative vertices.  See the example in Figure \ref{fig:Gor-sing}.

\begin{figure}[h]
\begin{center}
\includegraphics[scale=0.6]{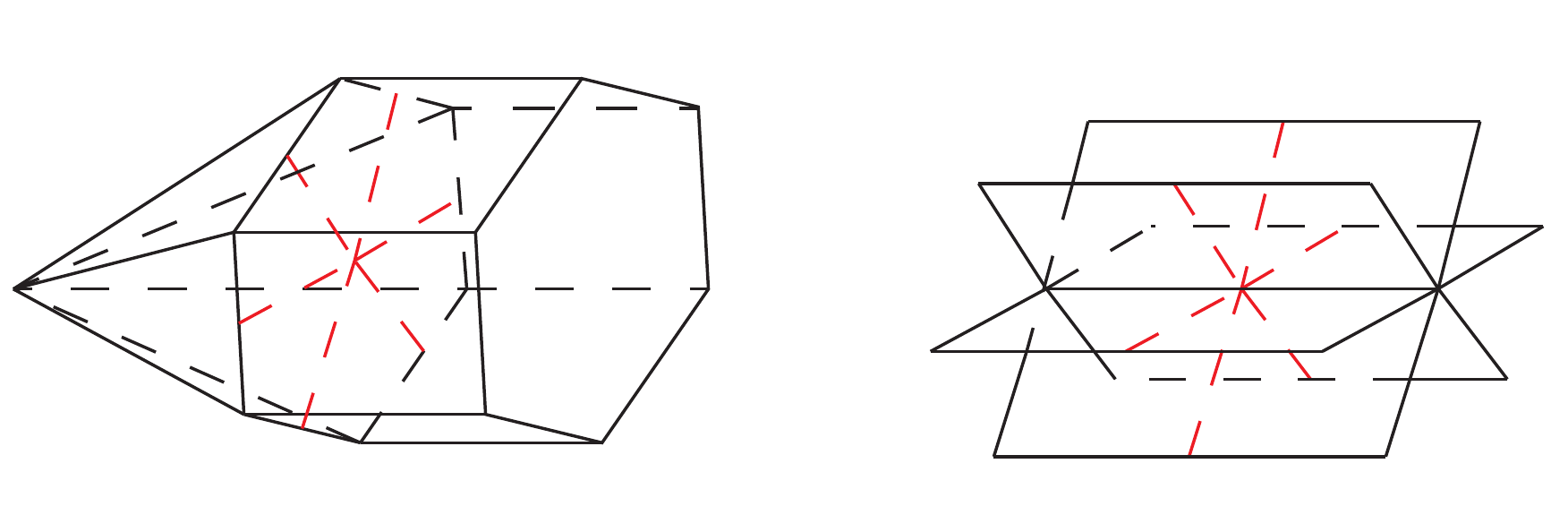}
\caption{An example of affine structure for toric Gorenstein singularity.  The right hand side (positive vertex) is the toric Gorenstein singularity and the left hand side is its mirror.}
\label{fig:Gor-sing}
\end{center}
\end{figure}

An interesting class of geometries is hyperconifold studied by physicists \cite{Davies1,Davies2}.  It gives an instance that a conifold transition (at several nodes simultaneously) is not mirror to a reverse conifold transition (but to a hyperconifold transition instead).  In this case $P$ is taken to be a parallelogram spanned by two vectors $v,w$.  It includes orbifolded conifolds as special cases.  

\begin{remark}
A Minkowski decomposition of the polytope $P$ gives a smoothing of the corresponding toric Gorenstein singularity \cite{altmann} (or a resolution of its mirror).  However it may not exist in general.  (In contrast a triangulation of $P$ into standard triangles which gives a resolution of the toric Gorenstein singularity always exists.)
\end{remark}

\section{Rational elliptic surfaces and the `12' Property} \label{sec:surf}

In this section we study tropical surfaces corresponding to rational elliptic surfaces with $A_n$ singularities.  A large part of this section is well-known to experts.  It is a reformulation of the works of \cite{Symington,LS} using the terminologies of \cite{GS07,CM1}.

First we introduce an easy generalization of the well-known `12' Property for non-convex polygons.  It is a special case of the `legal loops' in \cite{PR}.  Then we construct the corresponding tropical surfaces and obtain mirror pairs of symplectic rational elliptic surfaces with singularities from dual reflexive polygons.  We also generalize the construction for legal loops and relate with elliptic surfaces.  By using the classical result of Matsumoto \cite{Matsumoto}, it gives a topological proof of the generalized `12' property.

The significant work of Gross-Hacking-Keel \cite{GHK} developed mirror symmetry for log Calabi-Yau surfaces whose anti-canonical divisor is a nodal cycle of holomorphic spheres.  For rational elliptic surfaces the corresponding anti-canonical divisor is a smooth elliptic curve.

%{\color{red} Relation with ADE?}

\subsection{The `12' Property}

The following is a special case of the `12' Property for legal loops in \cite[Section 9.1]{PR} when the winding number is $1$ and there is no clockwise move in the legal loop.  (See Theoreom \ref{thm:12-legal}.)

\begin{prop} \label{prop:12}
Let $v_1, \ldots, v_m \in \Z^2-\{0\}$ be distinct primitive vectors, labeled in the counterclockwise manner, such that $\R_{\geq 0}\cdot\{v_1,\ldots,v_m\}=\R^2$, and the simplex $\{0,v_i, v_{i+1}\}$ does not contain any interior lattice point for every $i \in \Z/m\Z$.  Let $P$ be the union of all these simplices. 
Then
$$ 2 \, \mathrm{Area}(P) + \sum_{i\in\Z/m\Z} \det (u_{i-1},u_{i}) = 12 $$
where $u_i$ is the primitive vector in the direction of $v_{i+1}-v_i$.  
\end{prop}

\begin{proof}
First we consider the simplest case where $v_1=(1,0)$, $v_2=(0,1)$ and $v_3=(-1,-1)$.  It is easy to check the equality directly.  To get better geometric understanding, we consider the complete fan generated by these vectors, which gives the toric manifold $\bP^2$.  Then $2 \, \mathrm{Area}(P) = \sum_{i=1}^m\det(v_i,v_{i+1})=m=3$, the number of irreducible toric divisors.  Moreover $\det (u_{i-1},u_{i}) = c_1 \cdot D_i$ where $c_1 = \sum_{i=1}^m D_i$ and $D_i$ denotes the irreducible toric divisors corresponding to $u_i$.  We have $c_1 \cdot D_i = 2+D_i\cdot D_i$ where $2$ comes from the Euler characteristic of each irreducible toric divisor (which is topologically a sphere).  Hence the LHS of the above equals
$$ m + \sum_{i=1}^m (2+D_i^2) = 3m + \sum_{i=1}^m D_i^2 $$
which is equal to $12$ for $\bP^2$.

In general, by further subdividing $P$, we can always assume that $\det (v_i, v_{i+1}) = 1$.  Then the LHS always equals $3m + \sum_{i=1}^m D_i^2$ of the corresponding toric manifold.  We want to prove that it is always $12$.  Since every toric manifold can be obtained from $\bP^2$ by successively blowing up and down at toric points, it suffices to prove that $3m + \sum_{i=1}^m D_i^2$ is invariant under blow-up.  Blowing up at a toric point increases $m$ by $1$, adds a new exceptional curve of self-intersection $-1$, and decreases the self-intersection numbers of the two adjacent divisors by $1$.  Thus the total effect is $0$ and $3m + \sum_{i=1}^m D_i^2$ is $12$.
\end{proof}

\begin{defn}
Assume the notations in Proposition \ref{prop:12}.  The integer $\det (u_{i-1},u_{i})$ is called the order of the vertex $v_i$.
\end{defn}

Suppose $P$ in Proposition \ref{prop:12} is a convex polygon.  Regarding $P$ as the moment polygon of a toric orbifold, $\det (u_{i-1},u_{i})$ is the order of the isotropy group of the orbifold point corresponding to $v_i$.

\subsection{Affine rational elliptic surfaces with singularities}

\begin{prop}[Affine rational elliptic surfaces] \label{prop:not-conv}
Define $P$ (which is not necessarily convex) as in Proposition \ref{prop:12}.  There exist two tropical surfaces with singularities $(\cA,\cP)$ and $(\cA',\cP')$ associated to $P$ having the following properties.
\begin{enumerate}
\item Both $\cA$ and $\cA'$ are discs whose boundaries are affine circles, with the affine length being the number of lattice points contained in the boundary of $P$.  (An affine circle is a topological circle in $\cA$ whose intersection with $\cA-\Delta$ is an affine submanifold.)
\item The interior of $(\cA,\cP)$ contains an affine circle $C$ formed by some edges in $\cP$.  The interior of $(\cA',\cP')$ contains the polygon $P$ whose edges are affine line segments which form a part of $\cP'$.
\item The affine circle $C \subset \cA$ contains singular points (called `outer') which are in one-to-one correspondence with the edges of $P$.  The boundary of the polygon $P \subset \cA'$ contains singular points (called `inner') which are in one-to-one correspondence with the edges of $P$.  The multiplicity of such a singular point is equal to the affine length of the corresponding edge of $P$.
\item The open disc bounded by $C \subset \cA$ contains singular points (called `inner') which are in one-to-one correspondence with the corners of $P$.  The complement of the  polygon $P \subset \cA'$ contains singular points (called `outer') which are in one-to-one correspondence with the corners of $P$.  The multiplicity of such a singular point equals the determinant of the two primitive tangent vectors of $P$ at the corresponding corners (which is negative when the corresponding corner is non-convex).
\item The total number of singular points (counted with multiplicities) in $\cA$ or $\cA'$ equals $12$.
\end{enumerate}
\end{prop}

\begin{proof}
$(\cA,\cP)$ is defined as follows.  Take the triangles with corners $0,v_i,v_{i+1}$, and the parallelograms with corners $0,v_i,v_{i-1}-v_i,v_{i-1}$.  The triangles are glued to give the polygon $P$.  For the parallelograms the sides $\{0,v_i\}$ are glued to $\{v_{i}-v_{i+1},v_{i}\}$.  
The sides $\{v_i,v_{i-1}\}$ of the triangles are glued to the sides $\{0,v_{i-1}-v_i\}$ of the parallelograms.  See the left of Figure \ref{fig:ell-non-Fano-eg} for an example.

The fan structure at the vertex $0$ of a triangle is trivial.  At the vertex $v_i$ of a triangle, the fan is generated by $e_1,-e_1,e_2,-e_2$ and they are mapped to the primitive vector in the direction $v_{i-1}-v_i$, the primitive vector in the direction $v_{i+1}-v_i$, $v_i$ and $-v_i$ respectively.  At the vertex $v_i$ of a parallelogram, the fan is generated by $e_1,-e_1,-e_2$ and they are mapped to the primitive vector in the direction $v_{i-1}-v_i$, the primitive vector in the direction $v_{i+1}-v_i$ and $-v_i$ respectively.  This gives $(\cA,\cP)$.  Each edge of a triangle contains a singular point.  The singular points in the edges $\{0,v_i\}$ are called inner and those in $\{v_i,v_{i+1}\}$ are called outer.  It is a direct computation that the multiplicities are as stated in (3) and (4).

The sides $\{v_i,v_{i-1}\}$ of the parallelograms form the boundary of $\cA$ which is an affine circle.  The affine length is the sum of affine lengths of $v_i-v_{i-1}$, which is equal to the number of lattice points in the boundary of $P$.  This gives (1).  For (2), $C$ is given by the union of the sides $\{v_i,v_{i+1}\}$ of the triangles.  (5) follows from Proposition \ref{prop:12}.

$(\cA',\cP')$ is defined as follows.  Take the polygon $P$ with corners $v_i$, and the parallelograms with corners $0,v_i,v_{i-1}-v_i,v_{i-1}$.  For the parallelograms the sides $\{0,v_i\}$ are glued to $\{v_{i}-v_{i+1},v_{i}\}$.  
The sides $\{v_i,v_{i-1}\}$ of $P$ are glued to the sides $\{0,v_{i-1}-v_i\}$ of the parallelograms.  See the right of Figure \ref{fig:ell-non-Fano-eg} for an example.  

The fan at the vertex $v_i$ of $P$ is generated by the three vectors: $v_i$, the primitive vector along $v_{i+1}-v_i$ and that along $v_{i-1}-v_i$.  The fan at the vertex $v_i$ of a parallelogram is generated by $e_1,-e_1,-e_2$, and they are mapped to the primitive vector in the direction $v_{i-1}-v_i$, the primitive vector in the direction $v_{i+1}-v_i$ and $-v_i$ respectively.  Each edge of $P$ contains a singular point which is called inner.  Each edge $\{0,v_i\}$ of a parallelogram contains a singular point which is called outer.  Similarly one can verify the properties (1)-(5) for $(\cA',\cP')$.
\end{proof}

$(\cA,\cP)$ and $(\cA',\cP')$ are affine manifolds with toric polyhedral decompositions.  We call them to be affine rational elliptic surfaces.  They violate the positivity condition in Definition 1.54 of \cite{GS1} if $P$ is not convex.
Indeed $\cA$ and $\cA'$ are related by inversion; the inner singular points of $\cA$ correspond to the outer singular points of $\cA'$, and vice versa.
Figure \ref{fig:ell-non-Fano-eg} shows two examples for Proposition \ref{prop:not-conv} where $P$ is not convex.  One can glue local Lagrangian fibrations with monodromy $\left(\begin{array}{cc}
1 & 1 \\
0 & 1
\end{array}\right)$ and $\left(\begin{array}{cc}
1 & -1 \\
0 & 1
\end{array}\right)$ to obtain a Lagrangian fibration on a symplectic $4$-fold.  Note that negative multiplicities can only occur for a Lagrangian fibration but not for a holomorphic fibration.

\begin{figure}[htb!]
  \centering
   \begin{subfigure}[b]{0.7\textwidth}
   	\centering
    \includegraphics[width=\textwidth]{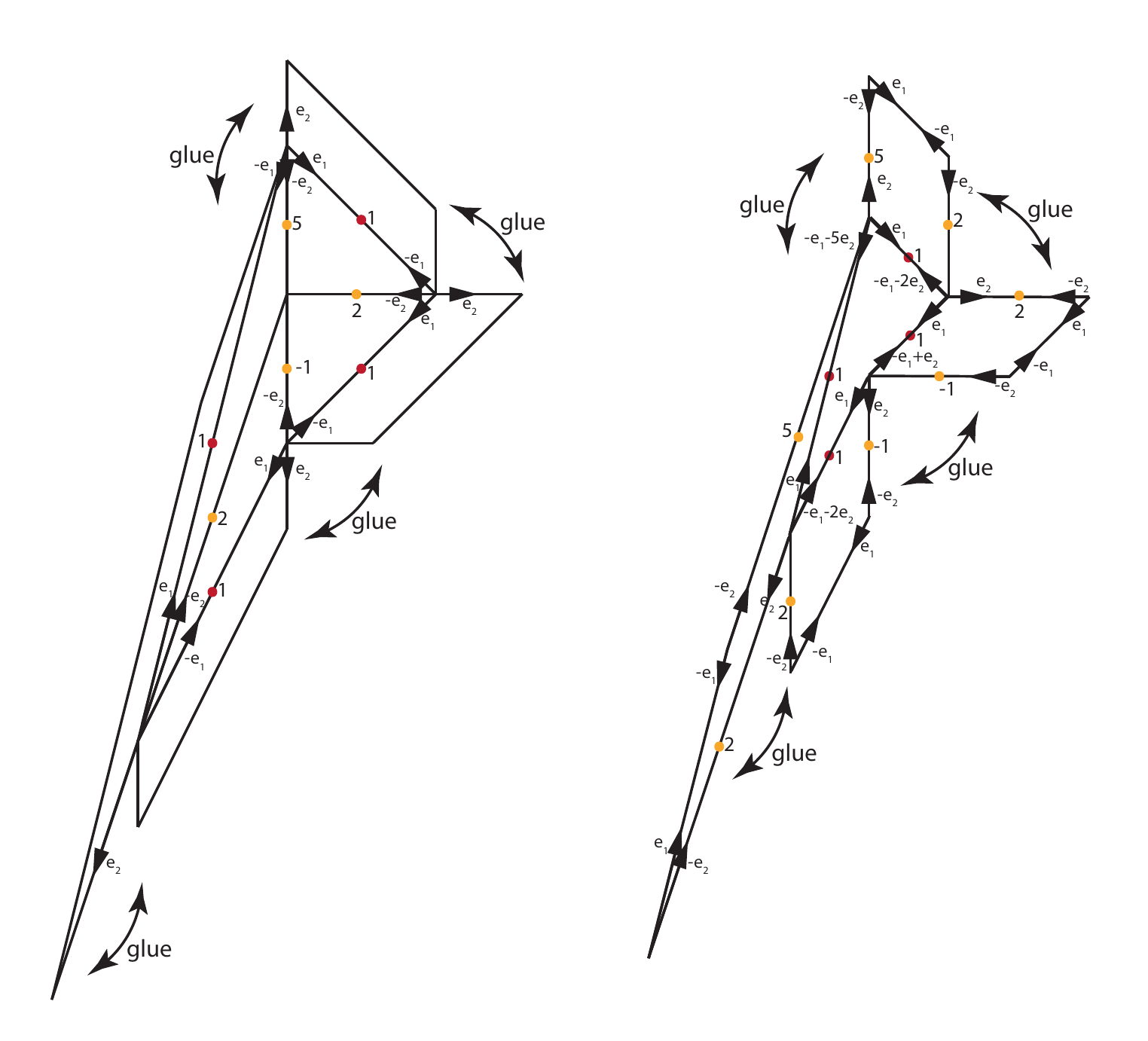}
    \caption{}
    \label{fig:ell-non-Fano-eg1}
   \end{subfigure}
   \hspace{10pt}
   \begin{subfigure}[b]{0.7\textwidth}
   	\centering
    \includegraphics[width=\textwidth]{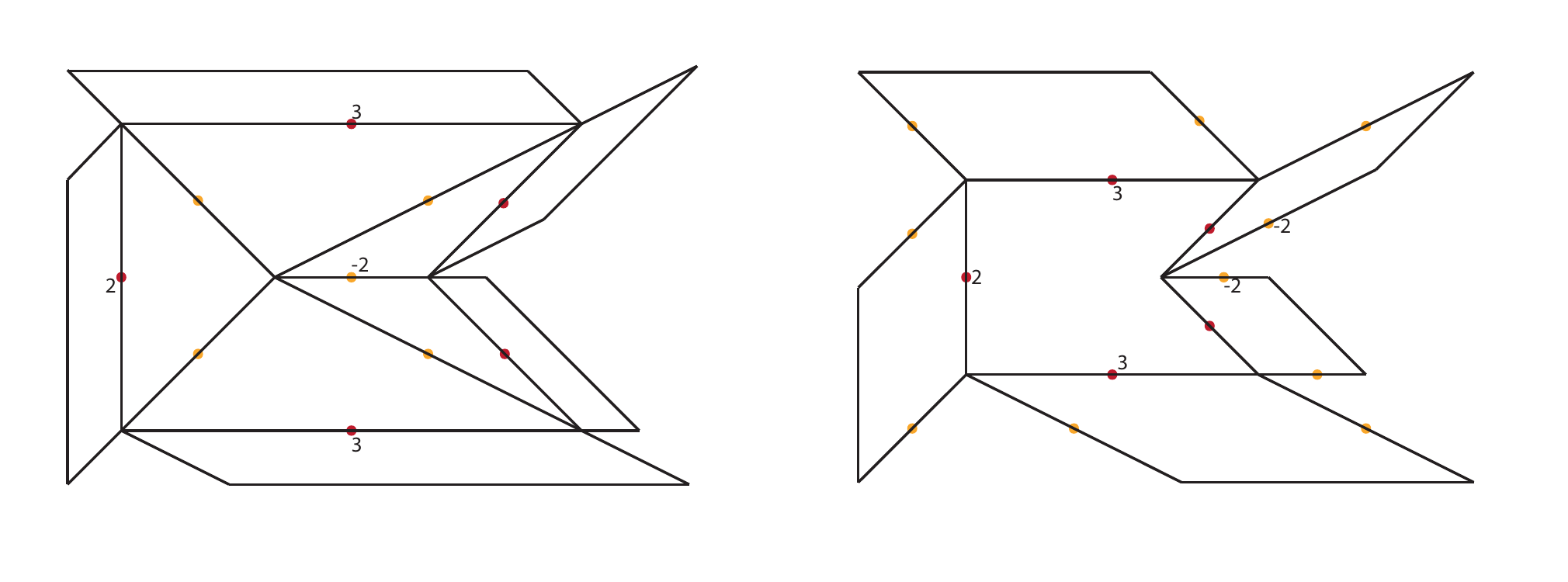}
    \caption{}
    \label{fig:ell-non-Fano-eg2}
   \end{subfigure}
\caption{Two examples of non-convex polygons.  (A) is constructed from the fan polytope of the Hirzebruch surface $\mathbb{F}_3$.  The `12' property still holds.}
  \label{fig:ell-non-Fano-eg}
\end{figure}

Note that the dual of a non-convex polygon in Proposition \ref{prop:12} is not a polygon.  It is a legal loop which is discussed in Section \ref{sec:legal}.  For this reason let's go back to the important case that $P$ is convex.  By definition $P$ is a reflexive polygon which has exactly one interior lattice point.  Its dual is again a reflexive polygon $\check{P}$.  Without loss of generality we assume that $u_i$ for $i=1,\ldots,m$ are the vertices of $P$.  Figure \ref{fig:ell-tor} lists all the affine surfaces constructed in this way.  (For simplicity we only list $\cA$ but not $\cA'$.)  

Each polygon in $(\cA'_{\check{P}},\cP'_{\check{P}})$ gives a piecewise linear function (unique up to addition of a linear function) supported on the dual fan at the corresponding vertex in $(\cA,\cP)$, and vice versa.  This gives a multivalued piecewise linear function $\phi$ on $(\cA,\cP)$.  Thus for a reflexive polygon we have a tropical manifold $(\cA,\cP,\phi)$.

\begin{figure}[h]
\begin{center}
\includegraphics[scale=0.5]{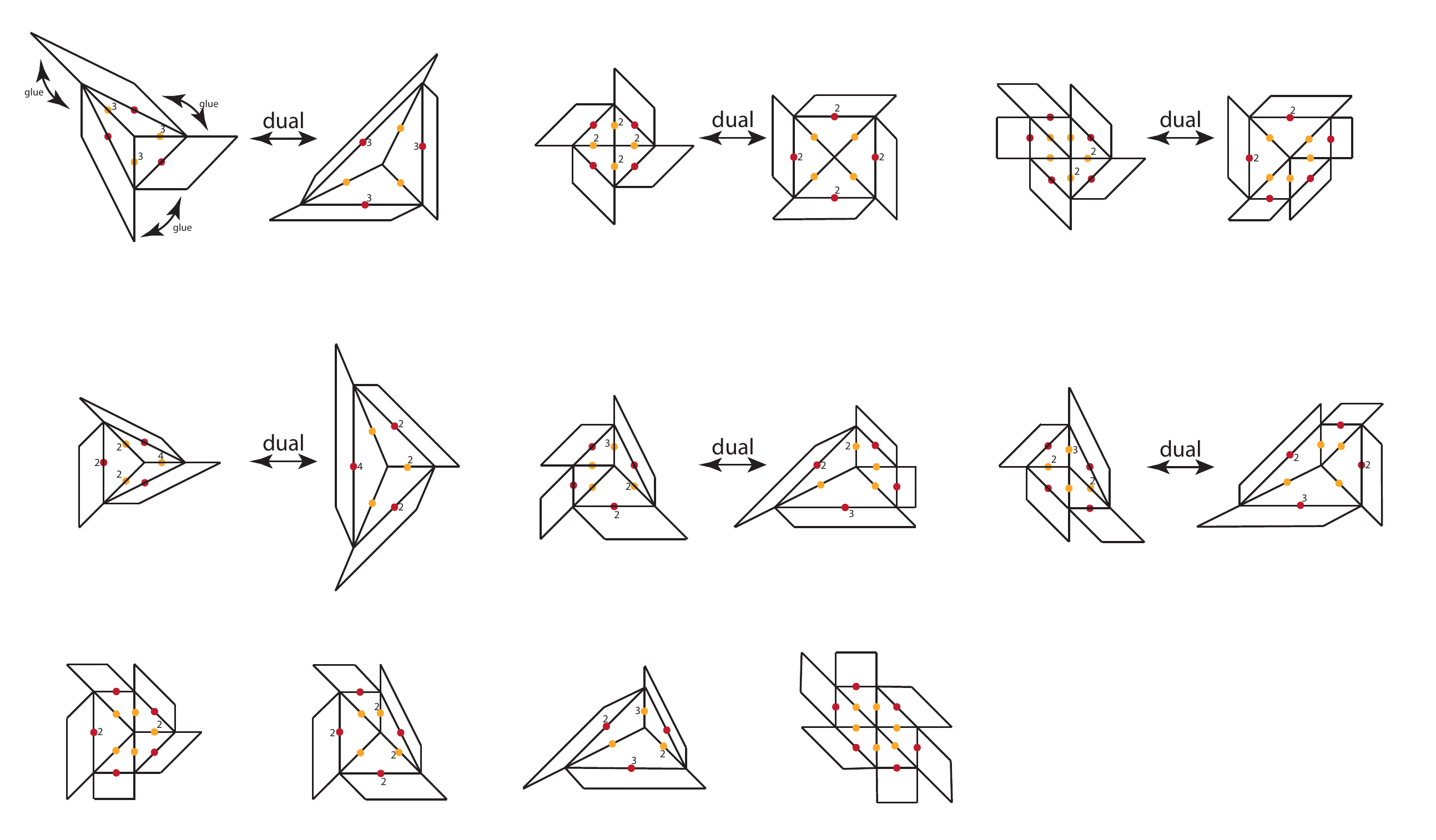}
\caption{Affine rational elliptic surfaces with $A_{k-1}$ singularities coming from reflexive polygons.  The dual polygons give mirror pairs.  The four in the last row are self-dual.}
\label{fig:ell-tor}
\end{center}
\end{figure}

By gluing in the Lagrangian fibrations on $A_{k-1}$ singularities, one obtains the following, which is a more precise version of Proposition \ref{thm:ell-intro}.

\begin{prop}[Symplectic rational elliptic surfaces] \label{thm:ell}
Each reflexive polygon $P$ corresponds to two symplectic rational elliptic surfaces with $A_k$ singularities $S$ and $S'$ together with Lagrangian fibrations.  Both $S$ and $S'$ satisfy the following properties.
\begin{enumerate}
\item The base of each fibration is topologically a closed disc.  The inverse image of the boundary is a symplectic torus.
\item The total number of interior singular fibers (counted with multiplicities) equals $12$.
\item The $A_k$ singularities can be divided into two groups, which are in one-to-one correspondence with non-standard corners of $P$ and non-standard simplices formed by $\{v_i,v_{i+1}\}$ respectively.  One has $k = A-1$ where $A$ is the area of the parallelogram spanned by primitive tangent vectors at the corner in the first case, or the area of the parallelogram spanned by $v_i,v_{i+1}$ in the second case.
\item The singular fibers of the Lagrangian fibration can be divided into two groups, which are in one-to-one correspondence with corners of $P$ and simplices formed by $\{v_i,v_{i+1}\}$ respectively.  The multiplicity equals $A$ defined above.
\end{enumerate}
For a pair of dual reflexive polygons $(P,\check{P})$, we have the mirror pairs $(S_P-D_{S_P},S_{\check{P}}'-D_{S_{\check{P}}'})$ and $(S_P'-D_{S_P'},S_{\check{P}}-D_{S_{\check{P}}})$ in the sense that the corresponding tropical geometries are Legendre dual to each other, where $D_{S_P},D_{S_{\check{P}}},D_{S_P'},D_{S_{\check{P}}'}$ denote the symplectic tori defined in (1).
\end{prop}

\begin{proof}
$T^*(\cA-\partial \cA-\Delta)/\Lambda^*$ gives a symplectic manifold, where $\Delta \subset \cA$ is the collection of singular points and $\Lambda\subset T(\cA-\partial \cA-\Delta)$ is the local system induced from the integral affine structure.  The Lagrangian fibration given in \eqref{eq:fib_An} has monodromy $\left(\begin{array}{cc}
1 & k \\
0 & 1
\end{array}\right)$
around the singular point $(0,0)$ in the base.  Thus its base is affine isomorphic to an affine $A_{k-1}$ singularity.  By using the action-angle coordinates, the Lagrangian fibration over a punctured neighborhood of $(0,0)$ is isomorphic to the fibration on $T^*(\cA-\Delta)/\Lambda$ around an affine $A_{k-1}$ singularity in $\Delta$.  Thus the local model can be glued in to give a Lagrangian fibration over $\cA-\partial \cA$.

$\partial \cA$ is an affine circle with length $L$.  We have the standard Lagrangian fibration $\pi_0:(\C^\times / L\Z) \times \C \to (\R/L\Z) \times \R_{\geq 0}$ (with the standard symplectic form).  A neighborhood of $(\R/L\Z) \times \{0\}$ in the base of $\pi_0$ is affine isomorphic to a neighborhood of $\partial \cA \subset \cA$.  Thus we can glue in a neighborhood of $\bS^1 \times (\R/L\Z)$ in $\pi_0$ and get a Lagrangian fibration over $\cA$.  This gives $S$.  The construction for $S'$ is similar.  The boundary divisor is the symplectic torus $\bS^1 \times (\R/L\Z)$.  Property (2)-(4) follow easily from Proposition \ref{prop:not-conv}.

It is easy to see that $\cA_P-\partial\cA_P$ and $\cA_{\check{P}}'-\partial\cA_{\check{P}}'$ are Legendre dual to each other.  Hence they form a mirror pair by the Gross-Siebert reconstruction program.
\end{proof}

\begin{remark}
We can take two such affine rational elliptic surfaces whose boundaries are affine circles with certain lengths, and glue them together (which are rescaled if necessary to match their boundary lengths) and obtain an affine sphere with singularities.  By the construction of \cite{CM1} it gives a symplectic K3 surface with $A_k$ singularities.  A pair of reflexive polygons $(P_1,P_2)$ and the dual $(\check{P}_1,\check{P}_2)$ produces a mirror pair of K3 surfaces with $A_k$ singularities.

In \cite{HZ1} a pair of dual reflexive polygons is used to construct an affine manifold corresponding to a Calabi-Yau manifold.  In the above proposition $P_1$ and $P_2$ are not necessarily dual to each other.  

Such a splitting is related to the conjecture of Doran-Harder-Thompson \cite{DHT} on mirror construction by gluing the Landau-Ginzburg mirrors of the components in a Tyurin degeneration.  See \cite[Section 7]{Kanazawa}.
\end{remark}

In Section \ref{sec:Schoen} we glue the product of an interval with an affine rational elliptic surface with another one to obtain a symplectic Calabi-Yau threefold with orbi-conifold singularities.

\subsection{Resolution and smoothing of rational elliptic surfaces with $A_k$ singularities} \label{sec:A_k-res}

Now we construct resolutions and smoothings of rational elliptic surfaces with $A_k$ singularities.  We shall stick with the example of a mirror pair given on the left of Figure \ref{fig:ell-tor-smoothing}.  All other rational elliptic surfaces associated to reflexive polygons have a similar construction.

\begin{figure}[h]
\begin{center}
\includegraphics[scale=0.6]{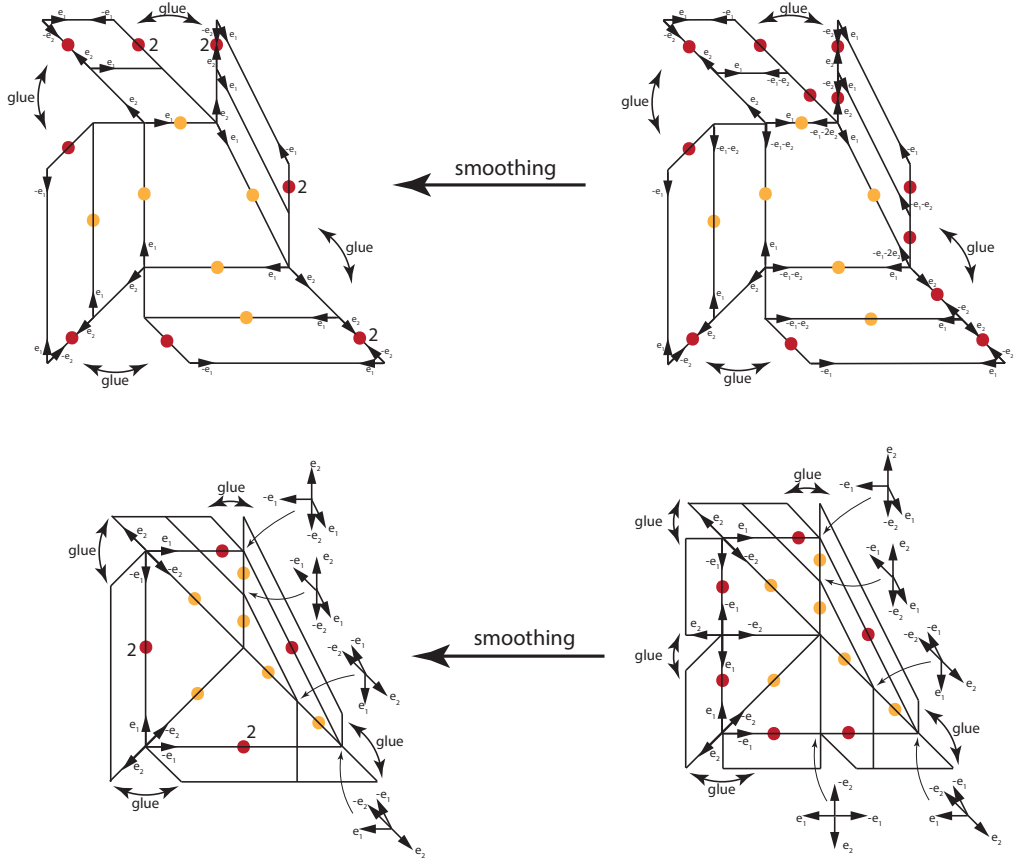}
\caption{An example of a mirror pair of elliptic surfaces with $A_n$ singularities.  The top left and bottom left figures show the mirror pair (this case the reflexive polygon is self-dual).  The top right and bottom right figures are their smoothings respectively.  The mirror of a smoothing of the top left figure is given by a resolution of the bottom left figure, which can be deduced by taking Legendre dual.
}
\label{fig:ell-tor-smoothing}
\end{center}
\end{figure}

Consider the second graph of the bottom row of Figure \ref{fig:ell-tor}, and compare with the bottom left of Figure \ref{fig:ell-tor-smoothing}.  In this example there are four inner and four outer singular points, and they have multiplicities $2,2,1,1$ respectively.  

First we take a smoothing of all the inner singular points, meaning that we separate each inner singular point with multiplicity $k>1$ into $k$ points with multiplicity $1$, and the $k$ points lie in the original monodromy-invariant affine line.  This is done by a certain refinement of the polyhedral decomposition and taking a suitable fan structure at each new vertex to match the monodromy.  There are several choices, and we have fixed one choice in the bottom left of Figure \ref{fig:ell-tor-smoothing}.

Then we take the Legendre dual, which gives a resolution of all the inner singular points for the dual polygon.  (In the example the polygon is self-dual.)  Note that each inner singular point with multiplicity $k>1$ is separated into $k$ points with multiplicity $1$ lying in distinct parallel monodromy-invariant affine line.  See the top left of Figure \ref{fig:ell-tor-smoothing}.

Now we either resolve or smooth out all the outer singular points.  Since resolutions can be obtained by taking Legendre dual of smoothings, we just show the smoothings here.  For the bottom left of Figure \ref{fig:ell-tor-smoothing}, we refine the reflexive polygon by taking cone over the lattice points lying in the relative interior of each boundary edge, and take a suitable fan structure at all the boundary lattice points to obtain the bottom right of Figure \ref{fig:ell-tor-smoothing}.  For the mirror shown in the top left of Figure \ref{fig:ell-tor-smoothing}, we refine the outer polygons if necessary and modify the fan structures at the vertices of the outer polygons to obtain the top right of Figure \ref{fig:ell-tor-smoothing}.

All other affine rational elliptic surfaces with $A_k$ singularities can be treated similarly.  We summarize by the following proposition.

\begin{prop}
All rational elliptic surfaces with $A_k$ singularities $S$ and $S'$ associated with reflexive polygons can be resolved and smoothed out symplectomorphically.  For each pair of dual reflexive polygons $P$ and $\check{P}$, a resolution of the rational elliptic surface $S_P$ (or $S_P'$) is mirror to a smoothing of $S'_{\check{P}}$ (or $S_{\check{P}}$) in the sense of Gross-Siebert.
\end{prop}

\subsection{Legal loops} \label{sec:legal}

Note that the dual of a non-convex polygon defined in Proposition \ref{prop:12} may not be a polygon.  A natural notion which is closed under duality is the legal loop defined in \cite{PR}.  The `12' Property holds for a general legal loop; the proof given in \cite{PR} uses modular forms.

\begin{defn}
A legal loop is a finite sequence of primitive vectors $\{v_1,\ldots,v_m\} \subset \Z^2$ with $v_i \not= v_{i+1}$ such that each triangle $\mathrm{Conv}\{0,v_i,v_{i+1}\}$, for $i=1,\ldots m$ where $v_{m+1}=v_1$, contains no interior lattice point.

A legal loop is said to be directed if $\det (v_i,v_{i+1})$ has the same sign for all $i$.
\end{defn}

The loop is formed by the union of edges connecting consecutive $v_i$.  Figure \ref{fig:legal-loop} and \ref{fig:legal-loop-2} show some examples.  For simplicity let's assume $(v_{i+1}-v_i) \nparallel (v_i - v_{i-1})$ for all $i$.  It is easy to see the following.

\begin{prop}
The dual (formed by taking the outward primitive vectors of the facets $\mathrm{Conv}\{v_i,v_{i+1}\}$ of the triangles $\mathrm{Conv}\{0,v_i,v_{i+1}\}$) of a legal loop is again a legal loop.  
\end{prop}

\begin{figure}[h]
\begin{center}
\includegraphics[scale=0.6]{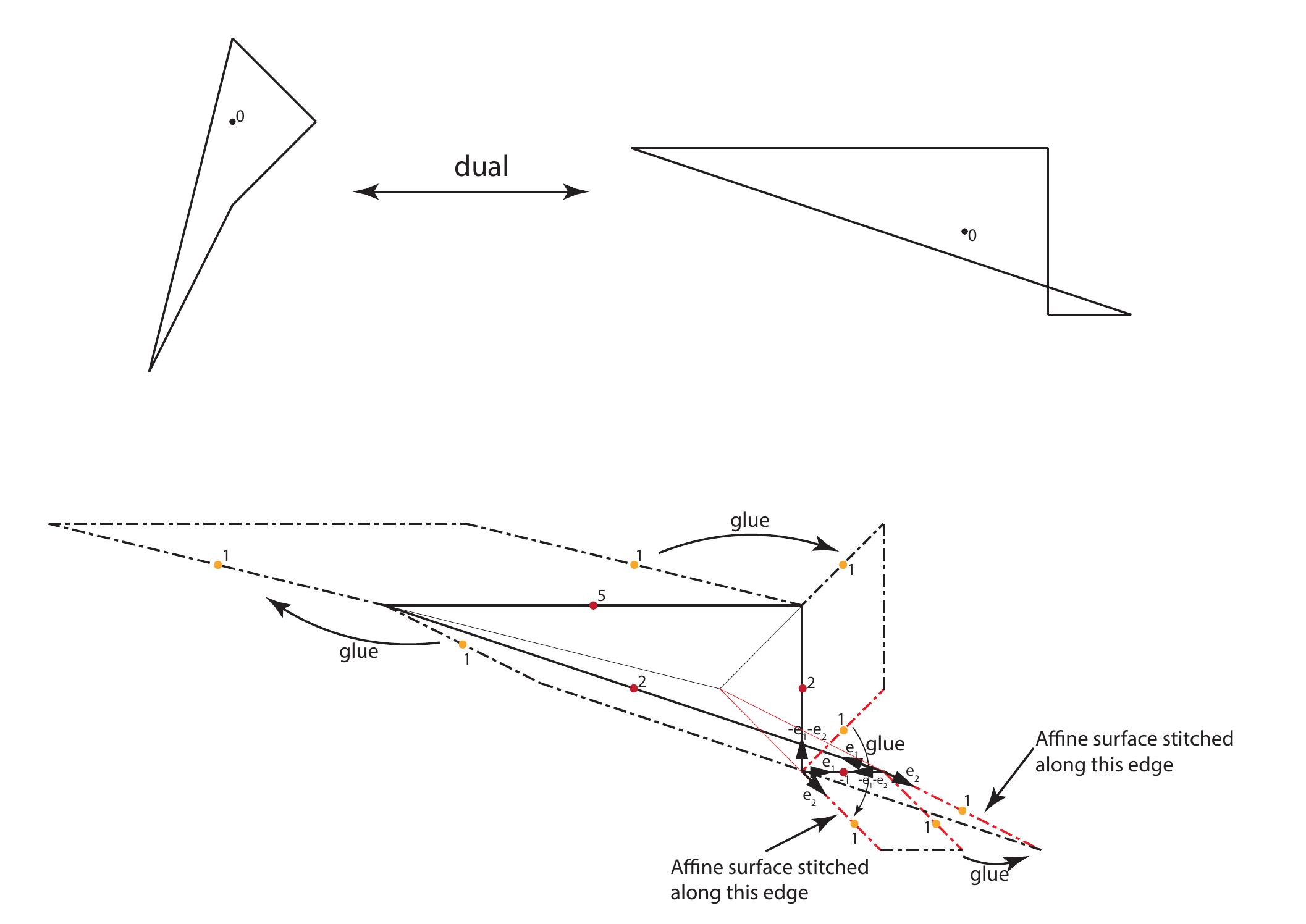}
\caption{A dual pair of legal loops.  The bottom figure shows a stitched affine manifold which is dual to the affine manifold given in the left of Figure \ref{fig:ell-non-Fano-eg1}.}
\label{fig:legal-loop}
\end{center}
\end{figure}

\begin{figure}[h]
\begin{center}
\includegraphics[scale=0.6]{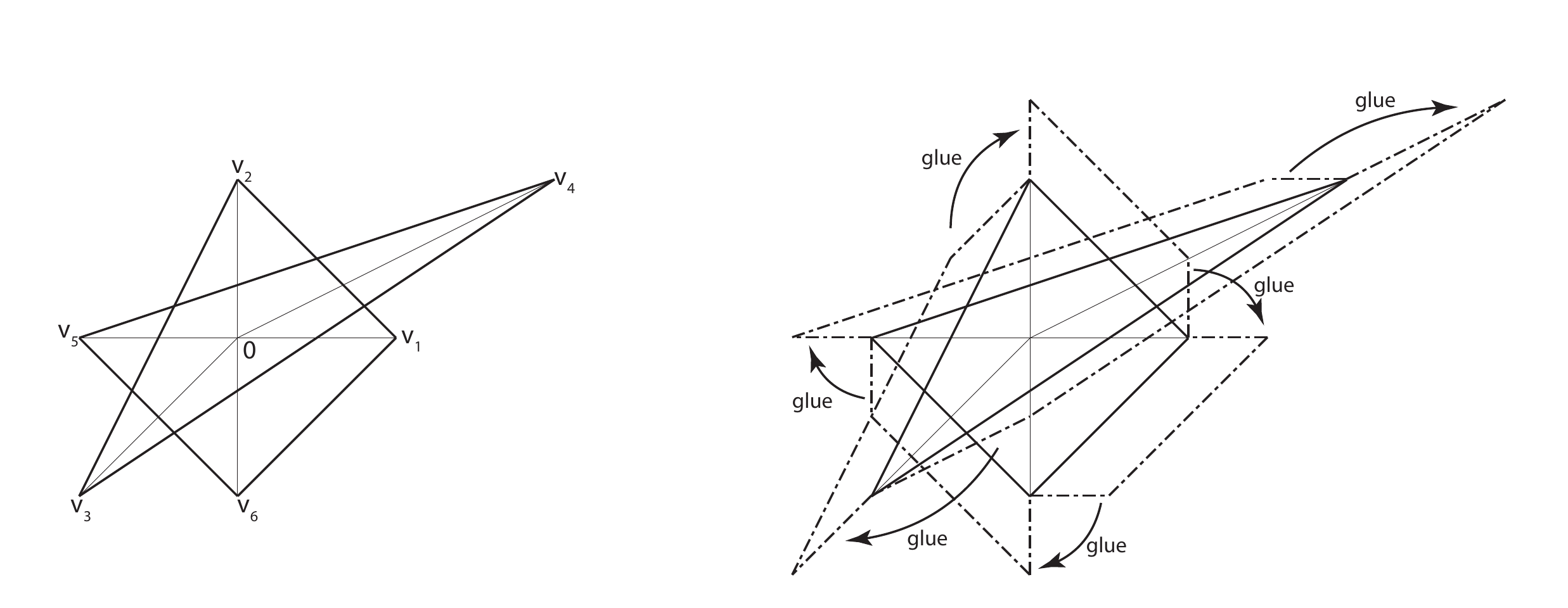}
\caption{An example of a legal loop which is directed and has winding number $2$.  It is associated with an integral affine manifold (with the origin removed).}
\label{fig:legal-loop-2}
\end{center}
\end{figure}

\begin{remark}
As shown in Figure \ref{fig:legal-loop}, the dual of a directed legal loop could be non-directed.
\end{remark}

The `12' Property for legal loops involves winding numbers.  It is stated as follows.

\begin{theorem}[\cite{PR}] \label{thm:12-legal}
For a legal loop with winding number $w \in \Z$ around the origin,
$$ \sum_{i\in\Z/m\Z} \det(v_i,v_{i+1}) + \sum_{i\in\Z/m\Z} \det (u_{i-1},u_{i}) = 12 w $$
where $u_i$ is the outward primitive normal vector of the triangle $\{0,v_i,v_{i+1}\}$.
\end{theorem}

Given a legal loop, we can associate it with a `folded' affine space.

\begin{defn}
A folded integral affine space is a smooth manifold $B_0$ (of dimension $n$) equipped with the following:
\begin{enumerate}
\item A codimension-1 submanifold $S \subset B_0$ (which may not be connected);
\item An integral affine structure on the closure of each connected component of $B_0-S$ (which is a manifold with boundary);
\item An isomorphism between the two integral affine structures along $S$ (which locally divides $B_0$ into two half-spaces).
\end{enumerate}
\end{defn}

Note that each of the integral affine structures at a point of $S$ in the above definition is the intersection of a lattice with a half-space, and the isomorphisms of lattices are required to preserve the half-spaces (and hence `fold' the half-spaces to each other).

To construct $B_0$, we do a similar construction as in the proof of Proposition \ref{prop:not-conv}: take the triangles with corners $0,v_i,v_{i+1}$, and the parallelograms with corners $0,v_i,v_{i-1}-v_i,v_{i-1}$.  Take the \emph{disjoint union} of these polygons, and then identify the sides $\{0,v_i\}$ among the triangles, identify the sides $\{0,v_i\}$ and $\{v_{i}-v_{i+1},v_{i}\}$ of the parallelograms, and identify the sides $\{v_i,v_{i-1}\}$ of the triangles and the sides $\{0,v_{i-1}-v_i\}$ of the parallelograms.  Finally we remove the points $0$ of the triangles, the mid-points of the sides $\{0,v_i\}$ and $\{0,v_{i-1}-v_i\}$ of the parallelograms.  This gives $B_0$ as an oriented surface.  

$S$ is defined to be the union of those sides $\{0,v_i\}$ of the triangles and $\{0,v_i\}$ of the parallelograms where $\det(v_{i-1},v_i)$ and $\det(v_i,v_{i+1})$ have different signs.  In particular if the legal loop is directed, $S = \emptyset$ and we do not have any folding.

As in Proposition \ref{prop:not-conv}, we have a similar local structure at each vertex, namely edges incident to the vertex are identified with rays in $\R^2$, and faces incident to the vertex are identified with cones in $\R^2$.  However they do not form a fan in general, namely some cones can intersect at their interior.  

The rays at the vertex $v_i$ of the triangles are generated by the four vectors: $v_i$, $-v_i$, the primitive vector along $v_{i+1}-v_i$ and that along $v_{i-1}-v_i$.  At the vertex $v_i$ of a parallelogram, the vector $-v_i$ is mapped to $-e_2$; the primitive vector in the direction $v_{i-1}-v_i$ and that in the direction $v_{i+1}-v_i$ are mapped to $\mathrm{sgn}(\det(v_{i-1},v_i)) e_1$ and $-\mathrm{sgn}(\det(v_{i-1},v_i)) e_1$ respectively.  This gives the structure of a folded integral affine space.  For directed legal loops this gives usual integral affine surfaces.  The proof of the following is similar to Proposition \ref{prop:not-conv} and hence omitted.

\begin{lemma}
The monodromy around $0$ is trivial.  The multiplicities of the singular points on the sides $\{0,v_i\}$ and $\{0,v_{i-1}-v_i\}$ of the parallelograms are given by $\det(u_{i-1},u_i)$ and $\det(v_{i-1},v_i)$ respectively, where $u_i$ is the outward primitive normal vector of the triangle $\{0,v_i,v_{i+1}\}$.
\end{lemma}

\begin{prop}[Torus fibration associated with legal loop]
Each legal loop is associated with a smooth torus fibration over $\bS^2$ with simple nodal singular fibers (with signs $\pm 1$).  Each vertex $v_i$ corresponds to $|\det(u_{i-1},u_i)|$ singular fibers of $\mathrm{sign}=\mathrm{sgn}(\det(u_{i-1},u_i))$ where $u_i$ is the outward primitive normal vector of the triangle $\{0,v_i,v_{i+1}\}$.  Each edge $\{v_i,v_{i+1}\}$ corresponds to $|\det(v_i,v_{i+1})|$ singular fibers of $\mathrm{sign}=\mathrm{sgn}(\det(v_i,v_{i+1}))$.  These are all the singular fibers.
\end{prop}

\begin{proof}
As in Section \ref{sec:A_k-res}, we can resolve or smooth out the $A_k$ singularities such that all of them become simple.  Denote the corresponding folded integral affine space by $(\tilde{B}_0,\tilde{S})$.  Take the torus bundle $T^*U/\Lambda^*$ where $U$ is a component of $\tilde{B}_0-\tilde{S}$ and $\Lambda$ denotes the integral structure in the tangent bundle.  The torus bundles for various $U$ are glued together according to the isomorphism along $S$.  This gives a torus bundle over $B_0$.  Then at each singular point with positive (or negative) sign, we glue in the fibration given by Equation \eqref{eq:fib_An'} (or multiplying the first component of Equation \eqref{eq:fib_An'} by $(-1)$).  This matches the monodromies of the singular points.  Now we have a torus fibration over an annulus.  Finally since the monodromy around $0$ and the outer boundary is trivial, we can glue in a trivial fibration around $0$ and $\infty$ to obtain a fibration over $\bS^2$.  This gives the fibration with the specified properties.
\end{proof}

%\begin{remark}
%When there is folding, the canonical symplectic structures on various $T^*U/\Lambda^*$ do not match along $S$, and so it only gives a topological surface.
%\end{remark}

\begin{proof}[A topological proof of the `12' Property (Theorem \ref{thm:12-legal})]
By a classical result of Matsumoto \cite[Theorem 1.1']{Matsumoto}, the total space of a good torus fibration is topologically either $V_k \# l (\bS^2 \times \bS^2)$ or $\bar{V}_k \# l (\bS^2 \times \bS^2)$, where $V_k$ is an elliptic surface with Euler characteristic $12k$.  The signature is to $8k$ and $-8k$ respectively.  On the other hand, the signature is equal to $(-2/3) (k_+ - k_-)$ where $k_+$ (or $k_-$) are the total numbers of positive (or negative) simple singular fibers.  Thus we see that $\sum_{i\in\Z/m\Z} \det(v_i,v_{i+1}) + \sum_{i\in\Z/m\Z} \det (u_{i-1},u_{i}) = 12 k$.  

To prove that $k$ equals the winding number $w$, it suffices to see that $k=0$ when $w=0$ (since we can concatenate legal loops).  When $w=0$, the torus fibration can be identified with the pulled back of a fibration over a contractible space, and hence is homotopic to a trivial bundle.  It follows that the signature is zero.
\end{proof}

%{\color{red} More than one interior lattice points?}

\section{Orbi-conifold transitions of Schoen's Calabi-Yau mirror pairs} \label{sec:Schoen}

Schoen's Calabi-Yau threefold \cite{Schoen} is (a resolution of) a fiber product of two rational elliptic surfaces over $\bP^1$.
In Section 9 of \cite{CM2}, Castano-Bernard and Matessi constructed conifold transitions of the Schoen's Calabi-Yau from the fan polytopes of toric blow-ups of $\bP^2$ (whose dual polytopes have standard vertices).

In this section we give a generalization of their construction.  Namely we construct orbi-conifold transitions of the Schoen's Calabi-Yau threefold and its mirror.  (Orbi-conifold points naturally occur when the vertices of a polygon are not standard.)   It gives compact examples of mirror pairs for orbi-conifolds.

Conceptually Schoen's Calabi-Yau degenerates to orbifolded conifolds as follows.  Take two elliptic surfaces with $A$-type singularities, and take their fiber product over $\bP^1$.  Suppose there is a point in $\bP^1$ such that both elliptic surfaces have singularities over that point.  Correspondingly the fiber product has an orbifolded conifold singularity.  In the following we use affine geometry to understand it more clearly.  The Gross-Siebert mirror (obtained by taking discrete Legendre transform) has generalized conifold singularities.

\subsection{General formulation}
We have explained local generalized and orbifolded conifolds in Section \ref{sec:SYZ}.  Below we define a notion which includes both singularities in a global geometry.
\begin{defn}
A complex (or symplectic) orbi-conifold is a topological space $X$ together with a discrete subset $S \subset X$ such that $X-S$ is a complex (or symplectic) orbifold, and for each $p \in S$ we have a homeomorphism of a neighborhood of $p$ to a neighborhood of the singular point in a local orbifolded or generalized conifold given in Section \ref{sec:SYZ}, and the homeomorphism is orbi-complex (or orbi-symplectomorphic) away from $p$.  A topological orbi-conifold is defined similarly without the complex (or symplectic) structure.
\end{defn}

\begin{defn}
A symplectic orbi-conifold transition of a symplectic manifold $X$ is another symplectic manifold $Y$ which is a resolution of a topological orbi-conifold $X_0$, where $X_0$ is a degeneration of $X$.  (The reverse process is also called an orbi-conifold transition.)
\end{defn}

\begin{remark}
A more natural definition should require $X_0$ to be a symplectic orbi-conifold.  Below we only construct $X_0$ as a topological orbi-conifold.  
To make it symplectic, we should use the technique of \cite{AAK} by Moser argument to construct a Lagrangian fibration on the local generalized conifold.
\end{remark}

Now we consider the affine setting.  The following definition is an orbifold generalization of \cite[Definition 6.3]{CM2}.

\begin{defn} \label{def:orbi-cfd}
An affine orbi-conifold is a polarized tropical threefold $(B,\cP,\phi)$ with the following properties.  
\begin{enumerate}
\item The discriminant locus $\Delta$ is a graph with vertices of valency 3 or 4.
\item Each edge of $\Delta$ has monodromy given by the matrix $\left(\begin{array}{ccc}
1 & 0 & 0 \\
k & 1 & 0 \\
0 & 0 & 1 \\
\end{array}\right)$
 in a suitable basis.
\item Every trivalent vertex of $\Delta$ has a neighborhood which is integral affine isomorphic to the orbifolded positive or negative vertex given in Section \ref{sec:orb_vert}. 
\item Every 4-valent vertex of $\Delta$ has a neighborhood which is integral affine isomorphic to the affine generalized or orbifolded conifolds given in Section \ref{sec:aff-loc}.
\end{enumerate}
\end{defn}

Since the affine generalized conifold point and the affine orbifolded conifold point are locally dual to each other (see Figure \ref{fig:loc-orb-cfd}), we immediately have the following.

\begin{prop} \label{prop:gen-dual-orb}
Let $(B,\cP,\phi)$ be an affine orbi-conifold, and $(\check{B},\check{\cP},\check{\phi})$ its Legendre dual.  Then the affine generalized conifold points of $(B,\cP,\phi)$ are in one-to-one correspondence with the affine orbifolded conifold points of $(\check{B},\check{\cP},\check{\phi})$, and vice versa.
\end{prop}

%Gluing local pieces... {\color{red} in order to glue to a global symplectic orbi-conifold, need a piecewise-smooth Lagrangian fibration on the local generalized conifold.  Construct by using \cite{AAK}.  Assume that the discriminant locus to be straight for simplicity.}

By gluing the local models of fibrations given in Section \ref{sec:SYZ} and \ref{sec:loc-aff}, we obtain a global orbi-conifold.  Combining with Proposition \ref{prop:gen-dual-orb}, we have the following.

\begin{prop}[Mirror symmetry of singularities] \label{prop:top-orb-cfd}
Given an affine orbi-conifold $(B,\cP,\phi)$, there exists a topological orbi-conifold $X$ and a torus fibration to $B$ with section whose discriminant locus and monodromies agrees with that of $B$.  Let $\check{X}$ be the topological orbi-conifold corresponding to the Legendre dual $(\check{B},\check{\cP},\check{\phi})$.  Then the generalized conifold points of $X$ are in one-to-one correspondence with the orbifolded conifold points of $\check{X}$, and vice versa.
\end{prop}

\begin{defn}
Given an affine orbi-conifold $(B,\cP,\phi)$, a smoothing of $(B,\cP,\phi)$ is a simple and positive polarized tropical threefold $(\tilde{B},\tilde{\cP},\tilde{\phi})$ such that the corresponding manifold $\tilde{X}$ is a (topological) smoothing of the orbi-conifold $X$.  
A resolution of $(B,\cP,\phi)$ is a simple and positive polarized tropical threefold such that its Legendre dual is a smoothing of the Legendre dual $(\check{B},\check{\cP},\check{\phi})$.
\end{defn}

Given a simple and positive polarized tropical threefold $(B,\cP,\phi)$, \cite{CM2} constructed a symplectic manifold with a Lagrangian fibration to $B$.   In particular an affine smoothing or resolution of an affine orbi-conifold corresponds to a symplectic manifold.  This gives a way to produce orbi-conifold transition via affine geometry.

\begin{remark}
Similar to Definition \ref{def:orbi-cfd}, we can allow vertices with higher valencies and define the notion of an affine variety with Gorenstein singularities by using the local models given in Section \ref{sec:Gor}.  The same technique can be applied to construct more general geometric transitions.
\end{remark}

%{\color{red} Any compact example for the above remark?}

\subsection{Examples of compact CY orbi-conifolds}

The following is a more precise formulation of the first part of Theorem \ref{thm:main}.

\begin{theorem}
Each pair of reflexive polygons $(P_1,P_2)$ corresponds to a compact orbifolded conifold $O^{(P_1,P_2)}$ and a compact generalized conifold $G^{(P_1,P_2)}$.  They have the following properties.
\begin{enumerate}
\item $O^{(P_1,P_2)}$ and $G^{(P_1,P_2)}$ contain orbifolded loci which are locally modeled by (an open subset of) $\C^2/\Z_k \times \C^\times \supset \{0\} \times \C^\times$.  
\item The components of the orbifolded loci of $O^{(P_1,P_2)}$ (or $G^{(P_1,P_2)}$) are in one-to-one correspondence with the non-standard corners of $P_i$ and edges of $P_i$ with affine length greater than one ($i=1,2$).  The multiplicity $k$ equals the index of the non-standard corner in the first case and equals the affine length of the edge in the second case.
\item $O^{(P_1,P_2)}$ is a degeneration of a Schoen's CY.  $G^{(P_1,P_2)}$ is a blow-down of a mirror Schoen's CY.
\end{enumerate}
\end{theorem}

Note that $O^{(P_1,P_2)}$ is mirror to $G^{(\check{P}_1,\check{P}_2)}$ (rather than $G^{(P_1,P_2)}$) where $\check{P}$ denotes the dual polygon of $P$.

%{\color{red} Can we answer a question in \cite{CM2}?}

The construction goes as follows.  The two reflexive polygons $P_1$ and $P_2$ are associated with two affine rational elliptic surfaces $\cA_r'$, $r=1,2$, with $A_k$ singularities (see Proposition \ref{prop:not-conv} for the notations).  Let $L_r$ be the affine lengths of the boundary of the polygons respectively.  We shall glue the affine threefolds $\cA_1' \times (\R/L_2\Z)$ and $(\R/L_1\Z)\times \cA_2'$ along their boundaries $\partial \cA_1' \times (\R/L_2\Z) \cong (\R/L_1\Z)\times \partial\cA_2' \cong (\R/L_1\Z) \times (\R/L_2\Z)$.  Then we get an affine $\bS^3$ with singularities, which corresponds to an orbifolded conifold degeneration of the Schoen's Calabi-Yau threefold.

\begin{remark}
If we glue them in the other way, namely $\partial \cA_1' \times (\R/\Z) \cong \partial\cA_2'\times (\R/\Z)$ with rescaling if necessary, then we get an affine $\bS^2 \times \bS^1$ which corresponds to $(K3 \textrm{ with } A_k \textrm{ singularities}) \times \bT^{2}$.
\end{remark}

For instance, take $P_1$ to be the polygon given in the second graph of the last row of Figure \ref{fig:ell-tor}, and $P_2$ the second graph of the first row of Figure \ref{fig:ell-tor}.  The gluing is given in Figure \ref{fig:SchoenCY-whole}. 

\begin{figure}[h]
\begin{center}
\includegraphics[scale=0.6]{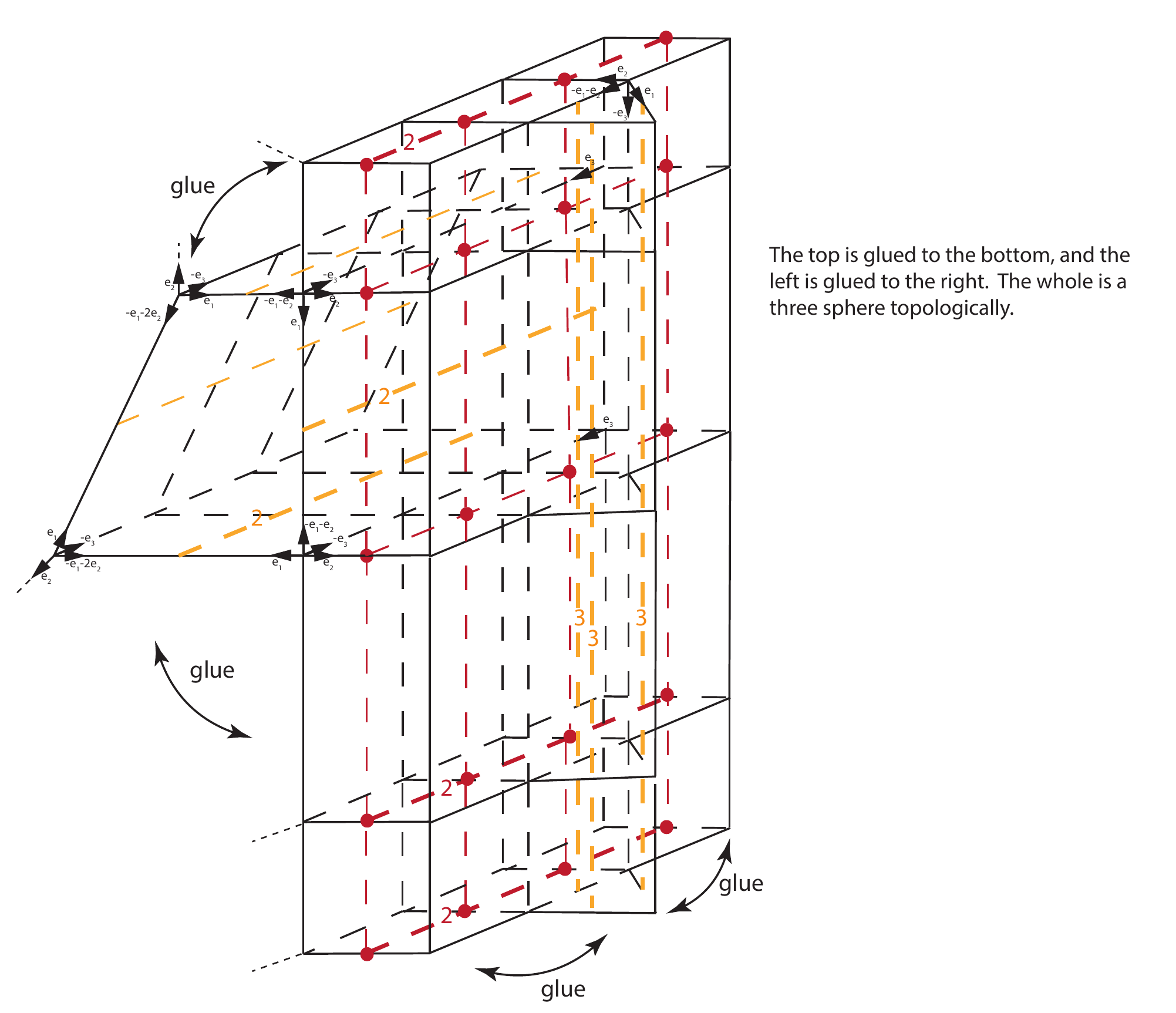}
\caption{An orbifolded conifold degeneration of Schoen's Calabi-Yau.  Each thick dot represents an orbifolded conifold singularity.}
\label{fig:SchoenCY-whole}
\end{center}
\end{figure}

The polyhedral decomposition consists of the following polytopes.  Let $P_r$ ($r=1,2$) consist of $m_r$ edges with affine lengths $l_k^{(r)}$ for $k=1,\ldots,m_r$ (the labeling of edges is in counterclockwise order).  Then $\sum_{k=1}^{m_r} l_k^{(r)} = L_r$.  We have the polytopes 
\begin{equation} \label{eq:hatP1}
\hat{P}_j^{(1)}:=[0,l_j^{(2)}] \times P_1
\end{equation} 
for $j=1,\ldots,m_2$,
\begin{equation} \label{eq:hatP2}
\hat{P}_i^{(2)}:=[0,l_i^{(1)}] \times P_2 
\end{equation}
for $i=1,\ldots,m_1$, and the cubes 
\begin{equation} \label{eq:C}
C_{ij}:=[0,l_i^{(1)}] \times [0,l_j^{(2)}] \times [0,1] 
\end{equation}
for $i=1,\ldots,m_1,j=1,\ldots,m_2$. 
The facet $\{0\} \times P_r$ (for $r=1,2$) of $\hat{P}_k^{(r)}$ is glued with the facet $\{l_{k-1}^{(r)}\} \times P_r$ of $\hat{P}_{k-1}^{(r)}$ for $k\in\Z/m_r\Z$.  The facet $\{0\} \times [0,l_j^{(2)}] \times [0,1]$ of $C_{ij}$ is glued with the facet $\{l_{i-1}^{(1)}\} \times [0,l_j^{(2)}] \times [0,1]$ of $C_{i-1,j}$ for $i\in\Z/m_1\Z$; the facet $[0,l_i^{(1)}] \times \{0\} \times [0,1]$ of $C_{ij}$ is glued with the facet $[0,l_i^{(1)}] \times \{l_{j-1}^{(2)}\} \times [0,1]$ of $C_{i,j-1}$ for $j\in\Z/m_2\Z$.  The facet $[0,l_i^{(1)}] \times [0,l_j^{(2)}] \times \{0\}$ of $C_{ij}$ is glued with the facet $[0,l_j^{(2)}] \times \partial_i P_1$ of $\hat{P}_j^{(1)}$ (where $\partial_i P_1$ denotes the $i$-th edge of $P_1$), and the facet $[0,l_i^{(1)}] \times [0,l_j^{(2)}] \times \{1\}$ of $C_{ij}$ is glued with the facet $[0,l_i^{(2)}] \times \partial_j P_2$ of $\hat{P}_i^{(2)}$.  The resulting total space is $\bS^3$ topologically.

Now we describe the fan structure at each vertex.  Every vertex of the polyhedral decomposition is either a vertex of $\hat{P}_i^{(1)}$ (for $i=1,\ldots,m_1$) or that of $\hat{P}_j^{(2)}$ (for $j=1,\ldots,m_2$).  It is always in the shape of a product of the fan of $\bP^1$ and that of a weighted projective plane.  Then the fan structure is taken to be the product of the fan of $\bP^1$ and the fan at the corresponding vertex of the affine rational elliptic surface $\cA_r$.

The discriminant locus $\Delta$ is taken to be the union of the lines 
\begin{equation} \label{eq:Delta_1}
[0,l_j^{(2)}]\times\{p\} \subset \partial \hat{P}_j^{(1)}
\end{equation}
for $j=1,\ldots,m_2$ and $p$ being the mid-point of an edge of $P_1$,
\begin{equation} \label{eq:Delta_2}
[0,l_i^{(1)}]\times\{p\} \subset \partial \hat{P}_i^{(2)}
\end{equation} 
for $i=1,\ldots,m_1$ and $p$ being the mid-point of an edge of $P_2$, 
\begin{equation} \label{eq:Delta_3} 
\{0\} \times [0,l_j^{(2)}] \times \{1/2\} \subset \partial C_{ij}
\end{equation}
and
\begin{equation} \label{eq:Delta_4}
[0,l_i^{(1)}] \times \{0\} \times \{1/2\} \subset \partial C_{ij}
\end{equation}
for $i=1,\ldots,m_1$ and $j=1,\ldots,m_2$.

This gives $(B,\cP)$.  One can directly verify the following and we omit the standard computations here.

\begin{lemma}
The polyhedral decomposition and the fan structures at vertices define an affine structure on $B-\Delta$.  The multiplicities of the discriminant loci defined by Equation \eqref{eq:Delta_1},\eqref{eq:Delta_2} are the lengths of the edges of $P_2$ and $P_1$ respectively; the multiplicities for Equation \eqref{eq:Delta_3}, \eqref{eq:Delta_4} are the orders of the $i$-th vertex of $P_1$ and the $j$-th vertex of $P_2$ respectively.  (The $i$-th vertex is adjacent to the $(i-1)$-th and $i$-th edges.)
\end{lemma}

The discriminant loci defined by Equation \eqref{eq:Delta_1} and \eqref{eq:Delta_2} correspond to the inner singular points of the affine surfaces $\cA_r'$.  (One can smooth out the inner singular points so that all of them are simple (having multiplicity $1$), see the top left of Figure \ref{fig:ell-tor-smoothing}.  Here we simply leave them as orbifolded singularities.)

The discriminant loci defined by Equation \eqref{eq:Delta_3} and \eqref{eq:Delta_4} correspond to the outer singular points of $\cA_r'$.  They intersect at $m_1\cdot m_2$ points, and these are the orbifolded conifold singularities.  For instance in Figure \ref{fig:SchoenCY-whole} there are $12$ orbifolded conifold points, $6$ of which are not the usual conifold points.

By Proposition \ref{prop:top-orb-cfd}, we have a topological orbifolded conifold $O^{(P_1,P_2)}$ corresponding to $(B,\cP)$.

We have the following compact affine generalized conifold $(\check{B},\check{P})$.  It is glued from $\cA^{\check{P}_1} \times (\R/\check{L}_2\Z)$ and $(\R/\check{L}_1\Z) \times \cA^{\check{P}_2}$, where $\check{P}_r$ is the dual polytope of $P_r$ and $\check{L}_r$ is the affine length of its boundary.  See Figure \ref{fig:SchoenCYmir-whole-4sides}.  (In this example $\check{P}_1$ is the same as $P_1$, and $\check{P}_2$ is given by the first graph of the first row of Figure \ref{fig:ell-tor}.)

\begin{figure}[h]
\begin{center}
\includegraphics[scale=0.6]{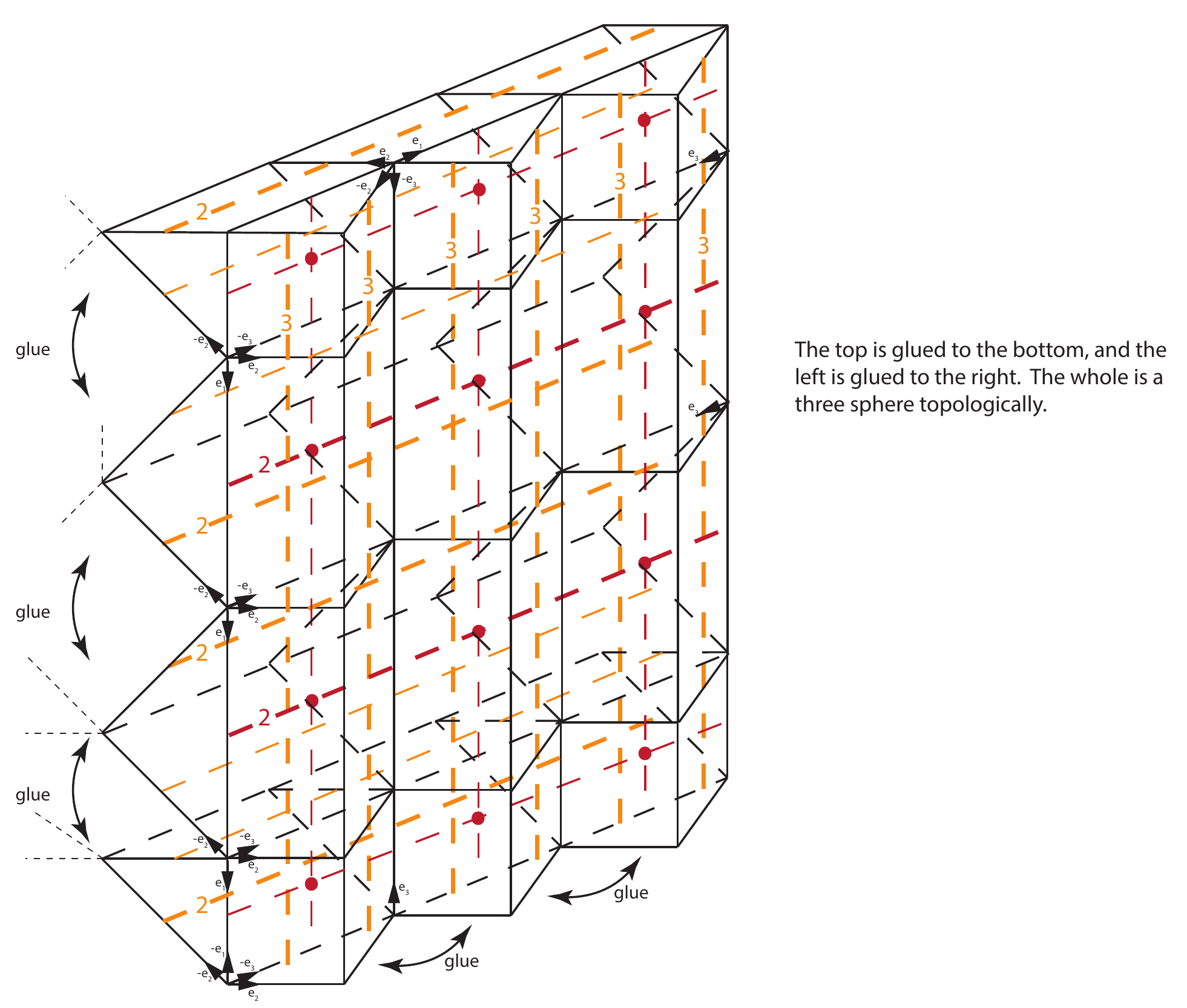}
\caption{The mirror obtained by taking Legendre dual.  Each thick dot represents a generalized conifold singularity.}
\label{fig:SchoenCYmir-whole-4sides}
\end{center}
\end{figure}

$(\check{B},\check{P})$ is obtained by gluing the prisms 
\begin{equation} \label{eq:hatT1}
\hat{T}^{(1)}_{ij}:=[0,\check{l}_j^{(2)}] \times T_i^{(1)} 
\end{equation}
and
\begin{equation} \label{eq:hatT2}
\hat{T}^{(2)}_{ij}:=[0,\check{l}_i^{(1)}] \times T_j^{(2)}
\end{equation}
along their boundaries, where $T_i^{(r)}$ are the triangles formed by the adjacent corners of $\check{P}_r$ for $r=1,2$.  The facet $\{\check{l}_j^{(2)}\} \times T_i^{(1)}$ of $\hat{T}^{(1)}_{ij}$ is glued with the facet $\{0\} \times T_i^{(1)}$ of $\hat{T}^{(1)}_{i,j+1}$; the facet $\{\check{l}_i^{(1)}\} \times T_i^{(2)}$ of $\hat{T}^{(2)}_{ij}$ is glued with the facet $\{0\} \times T_i^{(2)}$ of $\hat{T}^{(2)}_{i+1,j}$.  Denote the edges of the triangle $T_i^{r}$ by $\partial_0 T_i^{(r)}, \partial_1 T_i^{(r)},\partial_2 T_i^{(r)}$ corresponding to the $i$-th side of $\check{P}_r$ and its adjacent vertices  (in counterclockwise order) respectively.  The facet $[0,\check{l}_j^{(2)}] \times \partial_0 T_i^{(1)}$ of $\hat{T}^{(1)}_{ij}$ is glued with the facet $[0,\check{l}_i^{(1)}] \times T_j^{(2)}$ of $\hat{T}^{(2)}_{ij}$ by $\partial_0 T_i^{(1)} \cong [0,\check{l}_i^{(1)}]$ and $\partial_0 T_j^{(2)} \cong [0,\check{l}_j^{(2)}]$; the facet $[0,\check{l}_j^{(2)}] \times \partial_2 T_i^{(1)}$ of $\hat{T}^{(1)}_{ij}$ is glued with the facet $[0,\check{l}_j^{(2)}] \times \partial_1 T_{i+1}^{(1)}$ of $\hat{T}^{(1)}_{i+1,j}$; the facet $[0,\check{l}_i^{(1)}] \times \partial_2 T_j^{(2)}$ of $\hat{T}^{(2)}_{ij}$ is glued with the facet $[0,\check{l}_i^{(1)}] \times \partial_1 T_{j+1}^{(2)}$ of $\hat{T}^{(2)}_{i,j+1}$.  The resulting space is again $\bS^3$ topologically.

The dual discriminant locus $\check{\Delta}$ is given by the union of the lines 
\begin{align}
[0,\check{l}_j^{(2)}]\times\{p\} \subset \partial \hat{T}_{ij}^{(1)}, &\textrm{ $p$ is the mid-point of $\partial_1T_i^{(1)}$ or $\partial_2T_i^{(1)}$}, \label{eq:Deltacheck_1} \\
[0,\check{l}_i^{(1)}]\times \{p\} \subset \partial \hat{T}_{ij}^{(2)}, &\textrm{ $p$ is the mid-point of $\partial_1T_j^{(2)}$ or $\partial_2T_j^{(2)}$}, \label{eq:Deltacheck_2} \\
 [0,\check{l}_j^{(2)}]\times\{p\} \subset \partial \hat{T}_{ij}^{(1)}, &\textrm{ $p$ is the mid-point of $\partial_0T_i^{(1)}$}, \label{eq:Deltacheck_3} \\
[0,\check{l}_i^{(1)}]\times \{p\} \subset \partial \hat{T}_{ij}^{(2)}, &\textrm{ $p$ is the mid-point of $\partial_0T_j^{(2)}$}. \label{eq:Deltacheck_4}
\end{align}

The discriminant loci defined by Equation \eqref{eq:Deltacheck_1} and \eqref{eq:Deltacheck_2} correspond to the inner singular points of the affine surfaces $\cA^{\check{P}_r}$.  %One can resolve the inner singular points so that all of them are simple (having multiplicity $1$), see the bottom left of Figure \ref{fig:ell-tor-smoothing}.  Thus without loss of generality we can assume that these discriminant loci are simple.
The discriminant loci defined by Equation \eqref{eq:Deltacheck_3} and \eqref{eq:Deltacheck_4} correspond to the outer singular points of $\cA^{\check{P}_r}$.  They intersect at $m_1\cdot m_2$ points, and these are the generalized conifold singularities.  %This provides compact CY examples that orbifolded conifold points are mirror to generalized conifold points.

By Proposition \ref{prop:top-orb-cfd}, we have a corresponding topological generalized conifold $G^{(\check{P}_1,\check{P}_2)}$.  From the definition of $(\check{B},\check{\cP})$, one can write down a multivalued piecewise-linear function $\phi$ on $B$ such that $(\check{B},\check{\cP})$ is the Legendre dual of $(B,\cP,\phi)$.

\begin{remark}
If we take the polytope $P$ to be the moment map polytope of a toric blow-up of $\bP^2$ (or equivalently $\check{P}$ to be the fan polytope of a toric blow-up of $\bP^2$), then $B$ is a usual affine conifold.  This was constructed in \cite[Section 9.4]{CM2}.
\end{remark}

\subsection{Orbi-conifold transitions and mirror symmetry}

From the last subsection, we have a compact orbifolded conifold $O^{(P_1,P_2)}$ and a compact generalized conifold $G^{(P_1,P_2)}$ for each pair of reflexive polygons $(P_1,P_2)$.  In this subsection we construct their smoothings and resolutions in affine geometry.  By the Gross-Siebert reconstruction, they give a mirror pair of toric degenerations.  Due to discrete Legendre transform, a resolution of $O^{(P_1,P_2)}$ is mirror to a smoothing of $G^{(\check{P}_1,\check{P}_2)}$, and vise versa.  This completes the proof of Theorem \ref{thm:main}.

Since a resolution of the affine orbifolded conifold corresponding to $(\check{P}_1,\check{P}_2)$ is Legendre dual to a smoothing of the affine generalized conifold corresponding to $(P_1,P_2)$, and vice versa, it suffices to construct smoothings for $O^{(P_1,P_2)}$ and $G^{(P_1,P_2)}$.

\subsubsection{Smoothing of $O^{(P_1,P_2)}$.}  Recall that we have constructed a smoothing $\tilde{\cA}_{P_r}'$ for the affine elliptic surfaces $\cA_{P_r}'$, see the top right of Figure \ref{fig:ell-tor-smoothing}.  A smoothing of $O^{(P_1,P_2)}$ is given by gluing the two products of an affine circle with $\tilde{\cA}_{P_r}'$ for $r=1,2$.  See Figure \ref{fig:SchoenCY-whole-4sides-smoothing}.  

\begin{figure}[h]
\begin{center}
\includegraphics[scale=0.6]{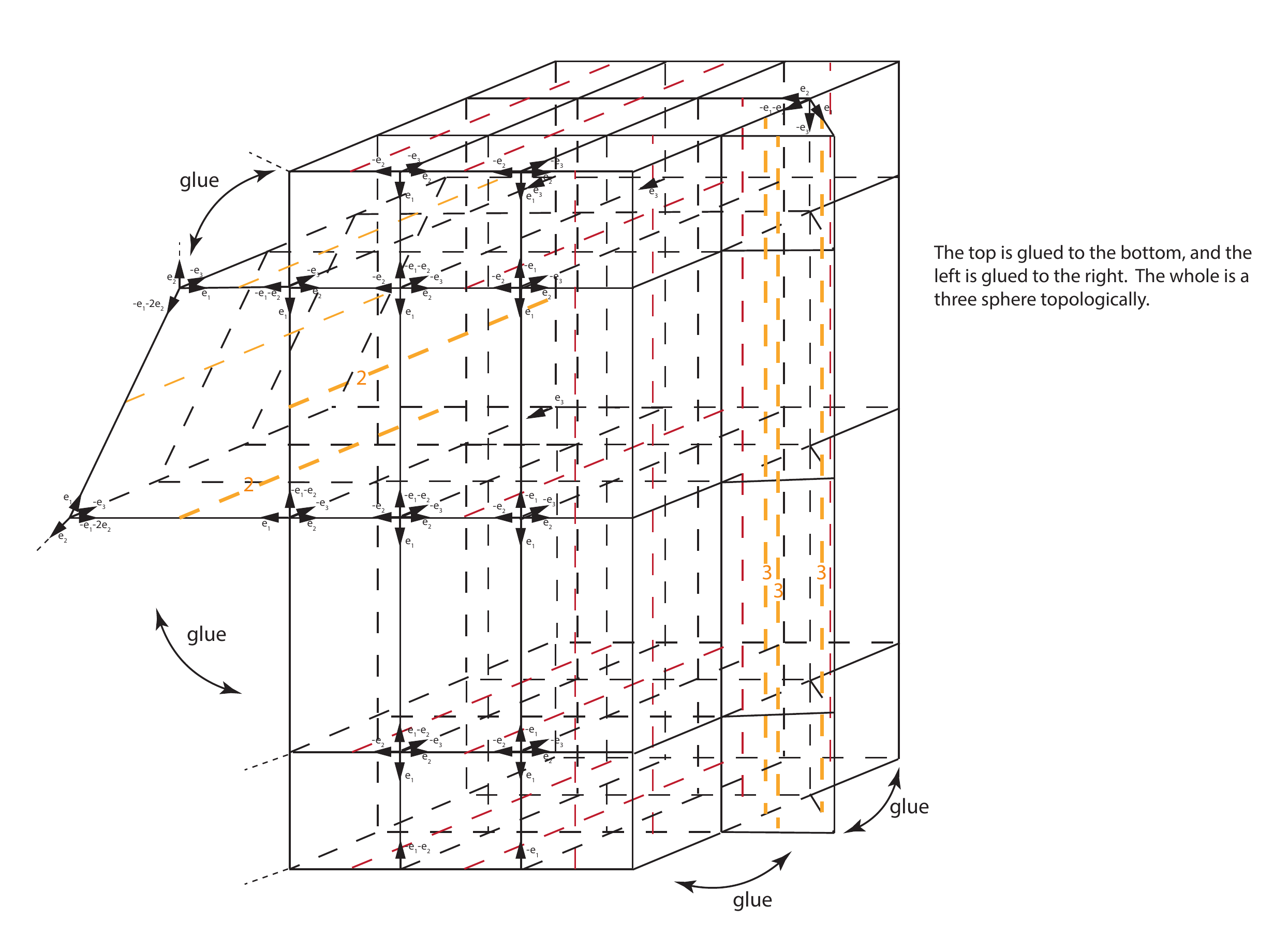}
\caption{A smoothing of an orbifolded conifold, which gives a Schoen's CY.  Its Legendre transform gives a resolution of the mirror generalized conifold (which is an orbi-conifold transition of the mirror Schoen's CY shown in Figure \ref{fig:SchoenCYmir-whole-4sides-smoothing}).}
\label{fig:SchoenCY-whole-4sides-smoothing}
\end{center}
\end{figure}

Let $K_r$, $r=1,2$, be the maximal multiplicity of the outer singular points of $\cA_{P_r}'$.  The polyhedral decomposition consists of $\hat{P}_j^{(1)}$ for $j=1,\ldots,m_2$ defined by Equation \eqref{eq:hatP1}, $\hat{P}_i^{(2)}$ for $i=1,\ldots,m_1$ defined by Equation \eqref{eq:hatP2}, and $(K_1+K_2)$ copies of $C_{ij}$ for $i=1,\ldots,m_1,j=1,\ldots,m_2$ defined by Equation \eqref{eq:C}.  They are glued as shown in Figure \ref{fig:SchoenCY-whole-4sides-smoothing}.  The fan structure at each vertex is taken to be the product of the fan of $\bP^1$ and the fan at the corresponding vertex of $\tilde{\cA}_{P_r}'$.  The discriminant locus is shown in the figure and we omit the detailed descriptions.  This gives a smoothing of $O^{(P_1,P_2)}$.  Its Legendre dual gives a resolution of $G^{(P_1,P_2)}$.

\subsubsection{Smoothing of $G^{(P_1,P_2)}$.} First we glue the products of an affine circle with a smoothing $\tilde{\cA}_{P_r}$ of $\cA_{P_r}$ (recall the bottom right of Figure \ref{fig:ell-tor-smoothing}) for $r=1,2$.  See Figure \ref{fig:SchoenCYmir-whole-4sides-smoothing}.  It consists of prisms $[0,1] \times \tilde{T}_i^{(1)}$ and $[0,1] \times \tilde{T}_j^{(2)}$, where $\tilde{T}_i^{(r)}$ for $r=1,2$ are the standard triangles formed by (the origin and) the adjacent boundary lattice points of $\check{P}_r$.

The resulting space is an affine conifold (with negative nodes).  It can be smoothed out by subdividing the rectangular facets containing the conifold singularities, and refining the polyhedral decomposition correspondingly.  Since all these negative nodes are contained in an affine plane, they can be simultaneously smoothed out by \cite[Theorem 8.7]{CM2}.  This gives a smoothing of $G^{(P_1,P_2)}$.  Its Legendre dual gives a resolution of $O^{(P_1,P_2)}$.  

As a result, we obtain orbi-conifold transitions of the Schoen's CY and its mirror.

\begin{figure}[h]
\begin{center}
\includegraphics[scale=0.6]{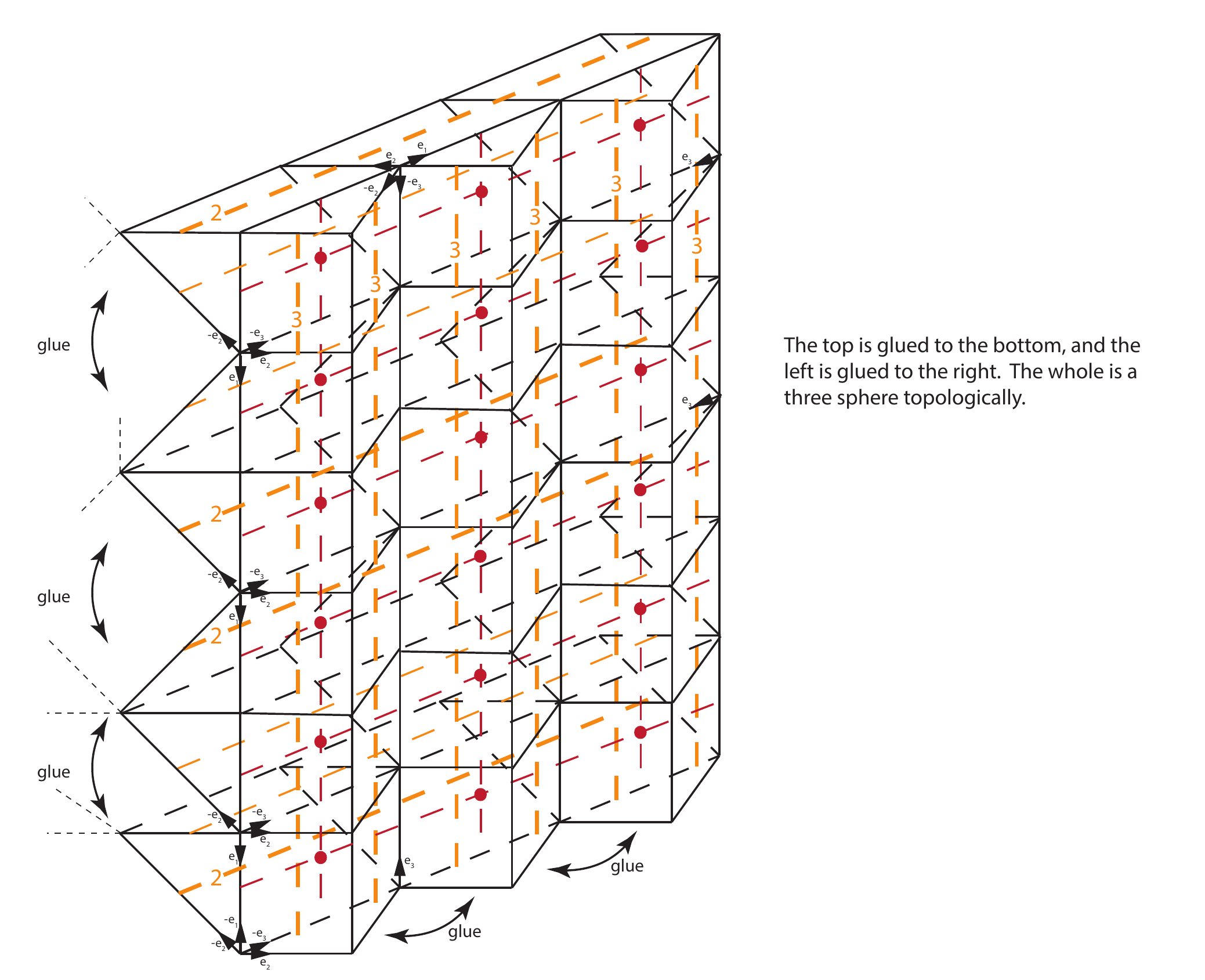}
\caption{A smoothing of the mirror.  The figure has (negative) conifold singularities which can be easily smoothed out by subdividing the $18$ rectangles containing the conifold singularities.  Its Legendre transform gives an orbi-conifold transition of Schoen's CY shown in Figure \ref{fig:SchoenCY-whole-4sides-smoothing}.}
\label{fig:SchoenCYmir-whole-4sides-smoothing}
\end{center}
\end{figure}

\begin{remark}
For convenience, in the above we first smooth out to a conifold, and then take further smoothing of the conifold.
There are other choices of subdivision which do not go through an conifold.  For instance see the left of Figure \ref{fig:loc-orb-con-res}.  In general orbi-conifold degenerations do not factor through ordinary conifold degenerations.
\end{remark}

\bibliographystyle{amsalpha}
\bibliography{geometry}

\end{document}